\documentclass{article}
\usepackage[utf8]{inputenc}
\usepackage{graphicx}
\usepackage[alphabetic]{amsrefs}
\usepackage{amssymb}
\usepackage{bbm}
\usepackage{xcolor}
\usepackage{hyperref}
\usepackage{lscape}
\usepackage{subfig}
\usepackage{titlefoot}
\usepackage{float}

\hypersetup{colorlinks}
\usepackage{amsthm}
\usepackage{enumitem}
\usepackage{amsmath}

\usepackage{ifthen}

\usepackage[export]{adjustbox}

\theoremstyle{plain}
\newtheorem{theorem}{Theorem}[section]
\newtheorem{lemma}[theorem]{Lemma}
\newtheorem{proposition}[theorem]{Proposition}

\newtheorem{corollary}[theorem]{Corollary}

\theoremstyle{definition}
\newtheorem{definition}[theorem]{Definition}
\newtheorem{example}[theorem]{Example}
\newtheorem{question}[theorem]{Question}

\theoremstyle{remark}
\newtheorem{remark}[theorem]{Remark}

\newtheorem*{To show}{To show}

\title{Conjugacy and Dynamics in Almost Automorphism Groups of Trees}
\author{Gil Goffer, Waltraud Lederle}

\let\phi\varphi
\renewcommand\emptyset\varnothing
\renewcommand{\restriction}{|}

\newcommand{\T}{\mathcal{T}}
\newcommand{\F}{\mathcal{F}}
\newcommand{\dT}{\partial \T}
\DeclareMathOperator\Aut{Aut}
\newcommand{\AutT}{\operatorname{Aut}(\T)}
\newcommand{\AAutT}{\operatorname{AAut}(\T)}

\newcommand{\Vertices}{\operatorname{Vert}}
\newcommand{\Edge}{\operatorname{Edge}}
\newcommand{\Fix}{\operatorname{Fix}}

\newcommand{\foris}[3]{
	\ifthenelse{\equal{#1}{}}%
	{}%
	{#1\colon}%
	\T \setminus #2 \to \T \setminus #3}%
\newcommand{\TwoT}{\T\setminus T}

\newcommand{\tp}[3]{[\overline{#1},#2,#3]}

\newcommand{\Tdk}{\T_{d,k}}
\newcommand{\Tdm}{\T_{d,m}}

\newcommand{\AutTdm}{\Aut(\Tdm)}

\DeclareMathOperator\AAut{AAut}

\newcommand{\AAutTdk}{\AAut(\Tdk)}

\newcommand{\V}{V}

\newcommand{\Sym}{\operatorname{Sym}}

\newcommand{\LL}{\mathcal{L}}

\newcommand{\Att}{\operatorname{Att}}
\newcommand{\Rep}{\operatorname{Rep}}
\newcommand{\Wan}{\operatorname{Wan}}

\DeclareMathOperator\OT{OT}
\newcommand{\BOT}{\mathrm{BOT}}

\DeclareMathOperator\eell{St}

\DeclareMathOperator\supp{Supp}
\newcommand{\Ell}{\mathcal{E}}
\newcommand{\Hyp}{\mathcal{H}}

\newcommand{\nh}{neighborhood}

\newcommand{\ehdecom}{EH decomposition}

\begin{document}
	
	\maketitle
	\unmarkedfntext{The first author was supported by a fellowship from the Ariane de Rothschild Women Doctoral Program. The second author was partially supported by Israel Science Foundation grant ISF 2095/15 and the Early Postdoc.Mobility grant number 175106 by the Swiss National Science Foundation. She also wants to thank the Weizmann Institute, where part of this work was completed, for its hospitality.}

	\begin{abstract}
		We determine when two almost automorphisms of a regular tree are conjugate. This is done by combining the classification of conjugacy classes in the automorphism group of a level-homogeneous tree by Gawron, Nekrashevych and Sushchansky and the solution of the conjugacy problem in Thompson's $V$ by Belk and Matucci. We also analyze the dynamics of a tree almost automorphism as a homeomorphism of the boundary of the tree.
	\end{abstract}

\section{Introduction}
	
	When are two elements of a group conjugate? Solving this question is a fundamental step in understanding a group. A classical framework in which it is addressed is the following setup. Given a finite group presentation $G= \langle S \mid R \rangle$, is there an algorithm that decides for two words with letters in $S$ whether they are conjugate or not? The answer is known to be ``yes" for Gromov hyperbolic groups, braid groups and others; but also many groups with unsolvable conjugacy problem are known.

In the current work we are looking at one of the most important examples in the theory of totally disconnected, locally compact groups, namely the almost automorphism group of a regular tree. We will give a precise definition of this group later. Roughly, its elements are equivalence classes of isomorphisms between subforests with finite complement.
The almost automorphism group of a regular tree was originally defined by Neretin \cite{n92} who studied its unitary representations. What makes it special is that it is the first known example of a simple, locally compact group not containing any lattices \cite{k99} \cite{bcgm12}. This result was recently strengthened by Zheng \cite{zhe19}, who showed that it is the first locally compact and compactly generated, non-discrete group not admitting any non-trivial IRS.

Let $\Tdk$ be a quasi-regular tree such that all but one vertices have valency $d+1 \geq 3$ and the remaining vertex has valency $k \geq 1$. Let $\AAutTdk$ be its almost automorphism group. There are two subgroups that are of specific importance.
The first is the automorphism group $\Aut(\Tdk)$ of $\Tdk$, which is open in $\AAutTdk$. The second is the Higman--Thompson group $V_{d,k}$, which is a countable dense subgroup $\AAutTdk$. For both of these subgroups, conjugacy has been solved.
Gawron, Nekrashevych and Sushchansky \cite{gns01} give a full description of conjugacy classes in $\Aut(\Tdk)$.
Barker, Duncan and Robertson \cite{bdr16} provide an algorithm solving the conjugacy problem in $V_{d,k}$ based on an algorithm described by Higman \cite{h74}.
The special case of $V=V_{2,2}$ has bean dealt with by Salazar-D\'iaz \cite{sd10} as well as Belk and Matucci \cite{bema14}. It is not hard to see that their solutions extend to $V_{d,k}$.
For $\AAutTdk$ we combine two different approaches. The first is the solution of conjugacy in $\Aut(\Tdk)$ via orbital types by Gawron, Nekrashevych and Sushchansky. The second is the solution of conjugacy in Thompson's $V$ via abstract strand diagrams by Belk and Matucci. We make heavy use of the notions of revealing pairs and rollings by Brin \cite{brin04} and Salazar-D\'iaz.

Closely related to conjugacy is dynamics. Namely, if $G$ is a group acting on a topological space $X$ and $g,h \in G$ are conjugate via an element $a \in G$ then the two dynamical systems $(X,g)$ and $(X,h)$ are topologically conjugate. In particular $a$ maps $g$-attracting points to $h$-attracting points, $g$-wandering points to $h$-wandering points, and so on. Recall that a wandering point is a point having a neighbourhood $U$ that is disjoint from $g^n(U)$ for all $n \geq 1$. For $G=\AAutT$ and $X=\dT$ the set of wandering points $\Wan(g)$ of every element $g$ is open and its closure is clopen and $g$-invariant. We can therefore write $g$ as a product $g=g_e g_h$, where $g_h|_{\overline{\Wan(g)}} := g|_{\overline{\Wan(g)}}$ and $g_h|_{\dT \setminus \overline{\Wan(g)}} := id$. A crucial observation is that determining whether $g$ and $h$ are conjugate can be reduced to separately checking whether $g_h$ and $h_h$ respectively $g_e$ and $h_e$ are conjugate, see Proposition  \ref{prop:conj_can_be_checked_on_ell_and_hyp_seperately}. This leaves us with two problems: Solving conjugacy for elements that do not have any wandering points, so-called \emph{elliptic} elements, and elements that act trivially outside the closure of the wandering points, we call them \emph{hyperbolic}. Le Boudec and Wesolek \cite{lbw19} previously divided tree almost automorphisms into elliptic elements and translations. What we call hyperbolic is a special case of their translations.

For a forest automorphism, we construct a labelled forest, which we call \emph{orbital type}. It is nothing else than the orbital type by Gawron, Nekrashevych and Sushchansky for a forest automorphism instead of a tree automorphism. Let $\F$ be a subforest of $\AAutTdk$ with finite complement.
The orbital type of a forest automorphism $\phi \in \Aut(\F)$ is the quotient forest $\langle \phi \rangle \setminus \F$, where each vertex in the quotient is labelled by the cardinality of its pre-image under the quotient map $\F \to \langle \phi \rangle \setminus \F$. Elliptic elements can be represented by forest automorphisms, see Lemma \ref{lem:ell=phiTT}, and we show that two elliptic elements $g$ and $h$ are conjugate if and only if the orbital types of such representatives are the same after removing a finite subgraph, see Theorem \ref{thm:conjugacy_elliptic}.

For a hyperbolic element, we show that it is conjugate to a sufficiently close element in the Higman--Thompson group $V_{d,k}$. What ``sufficiently'' means in this context leads us to the notion of revealing pairs by Brin \cite{brin04}. Having reduced ourselves to $V_{d,k}$ allows us to apply the results by Belk--Matucci.
They associate to every Higman--Thompson element a diagram, which we call a \emph{BM-diagram}, and prove that conjugacy is completely determined by this diagram.
A BM-diagram consists of three objects: a finite directed graph $D$ of a specific form, a cohomology class in $H^1(D,\mathbb{Z})$, and for every vertex an order on the edges adjacent to it. We prove that if two Higman--Thompson elements are close enough to one another, their reduced BM-diagrams differ only in these orders on the edges; and two hyperbolic elements in $\AAutTdk$ are conjugate if and only if sufficiently close Higman--Thompson elements have diagrams differing only in these edges' orders, see Theorem \ref{thm:hyperblolic_conjugacy}. We also explain how to read the dynamics of an element off its diagram (Theorem \ref{thm:read off dynamics}).
As an application we determine which hyperbolic elements are conjugate to a translation in $\AutT$, see Corollary~\ref{cor:hyperbolic conjugate to AutT}. The corresponding problem for elliptic elements seems to be complicated.

\begin{question}
 Find nice conditions under which an elliptic tree almost automorphism is conjugate to a tree automorphism.
\end{question}

Lastly, we show that an almost automorphism has open conjugacy class if and only if the set of wandering points is dense in $\dT$ (Corollary \ref{cor:open conj class}), and we determine closures of conjugacy classes for elliptic and hyperbolic elements.
Putting the elliptic and hyperbolic case back together seems to be surprisingly complicated.

\begin{question}
Let $g$ and $h$ be tree almost automorphisms that are neither elliptic nor hyperbolic. When is $g$ in the closure of the conjugacy class of $h$?
\end{question}
	
\section{Preliminaries} \label{sec: preliminaries}

	\subsection{Trees and their almost automorphisms}

All graphs in the current work are directed. All trees come with a root, which enables us to talk about children, descendants and ancestors of vertices.
Unless explicitly mentioned otherwise, edges in a tree point away from the root. For a tree $\T$ we denote its set of vertices by $\Vertices(\T)$ and its set of edges by $\Edge(\T)$.
Most of the time the tree at hand will be the $(d,k)$-quasiregular rooted tree $\Tdk$, whose root has $k$ children and whose other vertices all have $d$ children.

A \emph{caret} in a tree $\T$ is a finite subtree consisting of a vertex, the edges connecting it to its children and its children, see Fig. \ref{fig:caret}.
\begin{figure}[H]
    \centering
    \includegraphics[width=50mm]{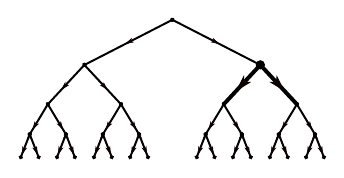}
    \caption{The thick lines indicate a caret.}
    \label{fig:caret}
\end{figure}
A subtree of $\T$ will be called \emph{complete} if it is a union of carets. Unless we explicitly state otherwise we will assume that complete subtrees contain the root, and as a consequence, all of the root's children.
When we form differences of complete subtrees, we always mean caret subtraction. This means that for subtrees $T'$ and $T$ of a tree $\T$ the difference $T' \setminus T$ consists of all carets of $T'$ that are not in $T$.
The maximal subtrees of $T' \setminus T$ we call \emph{components}.

Let $\T$ be an infinite tree. The \emph{boundary} of $\T$, denoted $\dT$, is as usual defined as the set of all infinite directed paths starting at the root.
Let $\T$ be an infinite tree and $x$ a vertex of $\T$. We denote by $\T_x$ the subtree of $\T$ with root $x$, and vertices being all descendants of $x$. Its boundary $\dT_x$ can be seen as a subset of $\dT$ in an obvious way, and all subsets of $\dT$ of the form $\dT_x$ form a basis of the topology of $\dT$.
If $x$ is not the root, we call such a basic open set a \emph{ball}, as a reference to the balls in the usual metric on $\dT$.

For a subtree $T$ of $\T$, we denote by $\LL T$ the set of leaves of $T$. Note that if $T$ is a finite complete subtree of $\T$, then $\{\dT_x\}_{x\in \LL T}$ is a finite clopen partition of $\dT$ into balls.

We denote the automorphism group of a tree $\T$ by $\AutT$, and for a finite subtree $T$ of $\T$ we write $\Fix(T)$ for the subgroup of $\AutT$ that fixes $T$ pointwise. Note that even though $\T$ is rooted, we will not assume that $\AutT$ necessarily fixes this root.

\begin{definition}
    Let $\T$ be an infinite tree without leaves and without isolated points in the boundary.
    An \emph{almost automorphism} of $\T$ is the equivalence class of a forest isomorphism $\foris{\phi}{T_1}{T_2}$, where $T_1$ and $T_2$ are complete finite subtrees of $\T$, and the equivalence relation is given by identifying two forest isomorphisms that agree outside of a finite set.
    
    \begin{figure}[H]
        \centering
        \includegraphics[width=\textwidth]{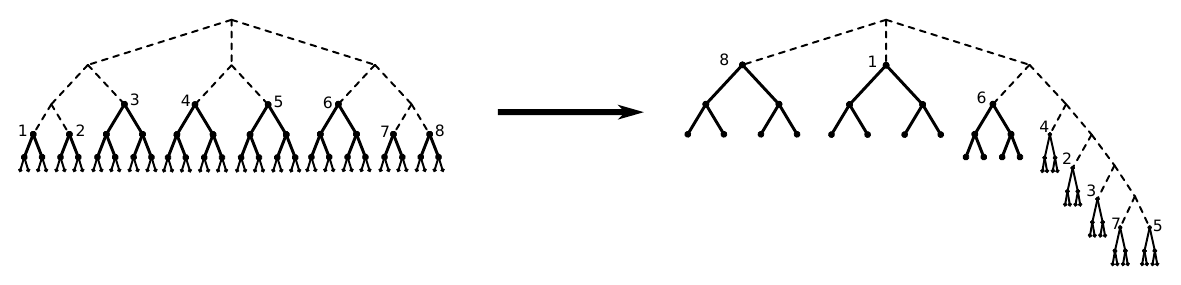}
        \caption{A representative of a tree almost automorphism. The dashed-lined trees are $T_1$ and $T_2$, and the numbers indicate how $\T\setminus T_1$ is mapped onto $\T\setminus T_2$.}
        \label{fig:tree almost automorphism}
    \end{figure}
\end{definition}

We refer to \cite{lb17} for a more detailed introduction to almost automorphisms.
The product of two almost automorphisms is formed by composing two representatives that can be composed as forest isomorphisms. 
Such representatives can always be found since for all almost automorphisms $g$ and $h$ and every large enough finite complete subtree $T$ there exist finite complete subtrees $T'$ and $T''$ and representatives $\foris{\phi}{T'}{T}$ and $\foris{\psi}{T}{T''}$ for $g$ and $h$, respectively.
The set of all almost automorphisms then forms a group, denoted $\AAutT$.
Every tree automorphism has an obvious interpretation as tree almost automorphism and it is not hard to see that with this interpretation $\Aut(\T) \leq \AAutT$.
This inclusion is used to define a group topology on $\AAutT$; we take $\{\Fix(T) \mid T \subset \T \text{ finite subtree}\}$ as basis of identity neighbourhoods in $\AAutT$. Clearly $\Aut(\T)$ is an open subgroup of $\AAutT$.

\begin{remark}\label{rem:iso_of_aauts}
    Let $\T$ and $\T'$ be trees such that there exist finite complete subtrees $T \subset \T$ and $T' \subset \T'$ and a forest isomorphism $\theta \colon \T \setminus T \to \T' \setminus T'$. Then $\theta$ induces an isomorphism $\AAutT \to \AAut(\T')$.
\end{remark}

We now turn our attention to a special subgroup of $\AAutT$.
A \emph{plane order} of $\T$ is a collection of total orders $\{<_x \mid x \in \Vertices(\T)\}$, where $<_x$ is a total order on the children of $x$. An almost automorphism is called \emph{locally order-preserving} if it has a representative $\foris{\phi}{T_1}{T_2}$ that maps the children of $x$ order-preservingly to the children of $\phi(x)$ for every vertex $x$ of $\T \setminus T_1$. This representative is then called \emph{plane order preserving}.

\begin{definition}
    The \emph{Higman--Thompson group} $V_{d,k}$ is the subgroup of $\AAutTdk$ consisting of all locally order-preserving almost automorphisms.
\end{definition}

It is not difficult to see that $V_{d,k}$ is dense in $\AAutTdk$ and that, up to conjugating with an element of $\Aut(\Tdk)$, it does not depend on the choice of the plane order. We can therefore fix a plane order of $\Tdk$ for the rest of the article.
For more information about Higman--Thompson groups, which are interesting far beyond being dense in $\AAutTdk$, consult \cite{h74}, \cite{b87} or \cite{cfp96}.

\paragraph{Translating boundary balls.}
Let $\T=\Tdk$. The group $\AAutT$ acts on $\dT$ in an obvious way. Recall that a boundary ball is a subset of the form $\dT_x\subset \dT$, where $x$ is not the root. Every boundary ball is the disjoint union of $d$ smaller boundary balls via replacing $x$ by its children. By induction, for any $m\equiv 1 \mod{d-1}$ it is also the union of $m$ balls.

\begin{lemma} \label{lem: translating balls} Let $\T=\Tdk$. Then following statements hold.
	\begin{enumerate}
	 \item \label{item:number of balls is preserved mod d-1}
	 Let $U \subset \dT$ be a clopen subset. Let $U= B_1 \sqcup \dots \sqcup B_{n_1} = C_1 \sqcup \dots \sqcup C_{n_2}$ be two partitions of $U$ into boundary balls. Then $n_1 \equiv n_2 \mod{d-1}$.
	    
	 \item \label{item:can map U1 to U2 while fixing W}
	 Let $U_1,U_2 \subset \dT$ be clopen non-empty proper subsets that can be partitioned into $n_1$ and $n_2$ boundary balls respectively. Let $W$ be a proper, possibly empty, clopen subset of $\dT \setminus (U_1 \cup U_2)$.
	    Then, there exists $g \in \AAutT$ fixing $W$ pointwise with $g(U_1)=U_2$ if and only if $n_1 \equiv n_2 \mod{d-1}$.
	\end{enumerate}
\end{lemma}

\begin{proof}
\begin{enumerate}
    \item Since every ball is a disjoint union of $d-1$ smaller balls, and since two balls are either disjoint or contained in one another, we can assume without loss of generality that the partition $C_1,C_2,\dots,C_{n_2}$ is a refinement of $B_1,B_2,\dots,B_{n_1}$. Under this assumption, it suffices to prove the case where $n_1=1$. Let $x$ be the vertex with $B_1=\dT_x$ and $x_1,\dots,x_{n_2}$ be the vertices with $C_i = \dT_{x_i}$. The fact $\T_x=\bigsqcup_i \T_{x_i}$ implies that $x_1,\dots,x_{n_2}$ are the leaves of a complete finite subtree rooted at $x$. Such a subtree exists only if $n_2 \equiv 1 \mod{d-1}$.

    \item We first prove the "only if"-direction. Let $g \in \AAutT$ with $g(U_1)=U_2$. Up to replacing the partition of $U_1$ by a refinement, we can assume that the $g$-image of each ball in the partition of $U_1$ is again a ball. This gives a partition of $U_2$ into $n_1$ balls. The fact that $n_1 \equiv n_2 \mod{d-1}$ now follows from Part \ref{item:number of balls is preserved mod d-1}. 

    For the "if"-direction, form two partitions $P_1$ and $P_2$ of $\dT$ into balls, satisfying: (a) each ball in $P_i$ is contained in either $U_i$, $W$ or $\dT \setminus (U_i \cup W)$; (b) $U_i$ is partitioned by $P_i$ into $n_i$ balls; and (c) $P_1$ and $P_2$ agree on $W$. By refining $P_1$ in $U_1$ (resp. $P_2$ in $U_2$) we can further assume that $n_1=n_2$.
    By Part \ref{item:number of balls is preserved mod d-1} the total number of balls in $P_1$ equals, $\mod{d-1}$, to the total number in $P_2$.
    Refine the partitions to make them have the same total number of parts, without affecting properties (a),(b) and (c). Indeed, this can be done by refining $P_i$ only over $\dT \setminus (U_i \cup W)$, which is non-empty by assumption.
    We are now ready to construct $g$. Let $T_1,T_2$ be the complete finite subtrees of $\T$ that correspond to the partitions $P_1$ and $P_2$ respectively. Take $g$ to be the almost automorphism induced by $\foris{\phi}{T_1}{T_2}$, mapping $U_1$ to $U_2$, $\dT \setminus (U_1 \cup W)$ to $\dT\setminus (U_2 \cup W)$ and fixing $W$ pointwise.
\end{enumerate}
\end{proof}

\subsection{Tree pairs}\label{subsec: tree pairs}

Historically, tree pairs were defined before tree almost automorphisms.

\begin{definition}\label{def:tree_pair}
    A \emph{tree pair} consists of two finite complete subtrees $T_1$ and $T_2$ of $\T$ together with a bijection $\kappa \colon \LL T_1 \to \LL T_2$ between their leaves. We denote it by $[\kappa,T_1,T_2]$.
\end{definition}

\begin{remark}\label{rem:kinds_of_leaves}
    Let $T_1$ and $T_2$ be two complete finite subtrees of $\T$. There are three different kinds of leaves of $T_1$, namely
    \begin{enumerate}
        \item leaves of $T_1$ that are also leaves of $T_2$, these are called \emph{neutral leaves}.
        \item leaves of $T_1$ that are interior vertices of $T_2$. They are roots of components of $T_2 \setminus T_1$; and
        \item leaves of $T_1$ that do not belong to $T_2$ at all. They are leaves of components of $T_1 \setminus T_2$.
    \end{enumerate}
    The analogous statement holds for leaves of $T_2$.
\end{remark}

We wil often consider $\kappa$-orbits in the leaves of $T_1$ and $T_2$.

\begin{definition}\label{def:maximal_chains}
    Let $P=[\kappa,T_1,T_2]$ be a tree pair. Let $x_0,\dots,x_n \in \LL T_1 \cup \LL T_2$. We call $(x_0,\dots,x_n)$ a \emph{maximal chain} of $P$ if it is an orbit under the partial action of $\kappa$. In other words $x_{i}=\kappa(x_{i-1})$ for $i=1,\dots,n$ and either
    \begin{enumerate}
        \item  $x_0 \notin \LL T_2$ and $x_n \notin \LL T_1 $; or
        \item  $\kappa(x_n)=x_0$.
    \end{enumerate}
    A maximal chain is called
    \begin{enumerate}
        \item an \emph{attractor chain}, and $x_n$ an \emph{attractor of period $n$}, if $x_n$ is a descendant of $x_0$;
        \item a \emph{repeller chain}, and $x_0$ a \emph{repeller of period $n$}, if $x_0$ is a descendant of $x_n$;
        \item a \emph{periodic chain} and each of $x_0,\dots,x_n$ a \emph{periodic leaf} if $x_0=\kappa(x_n)$; and
        \item a \emph{wandering chain}, and $x_0$ a \emph{source} and $x_n$ its corresponding \emph{sink}, if $x_0 \notin T_2$ and $x_n \notin T_1$.
    \end{enumerate}
\end{definition}

In Definition \ref{def:maximal_chains} we did not give a name to maximal chains that start at the root of a component and end in a vertex that is not their descendant or vice versa. This is because we prefer to consider tree pairs that do not have these kinds of maximal chains, as in the following definition due to Brin \cite{brin04}.

\begin{definition}
    Let $P=[\kappa,T_1,T_2]$ be a tree pair. It is called a \emph{revealing pair} if
    \begin{enumerate}
        \item every component of $T_1 \setminus T_2$ contains a (unique) repeller; and
        \item every component of $T_2 \setminus T_1$ contains a (unique) attractor.
    \end{enumerate}
\end{definition}

\begin{example}
    Figure \ref{fig:revealing pair} shows an example of a revealing pair. The gray tree is the common tree $T_1 \cap T_2$. Attractors and repellers are underlined, periodic leaves are circled, and a half moon marks the root of a component.
  	\begin{figure}
    \centering
    \includegraphics[scale=0.8]{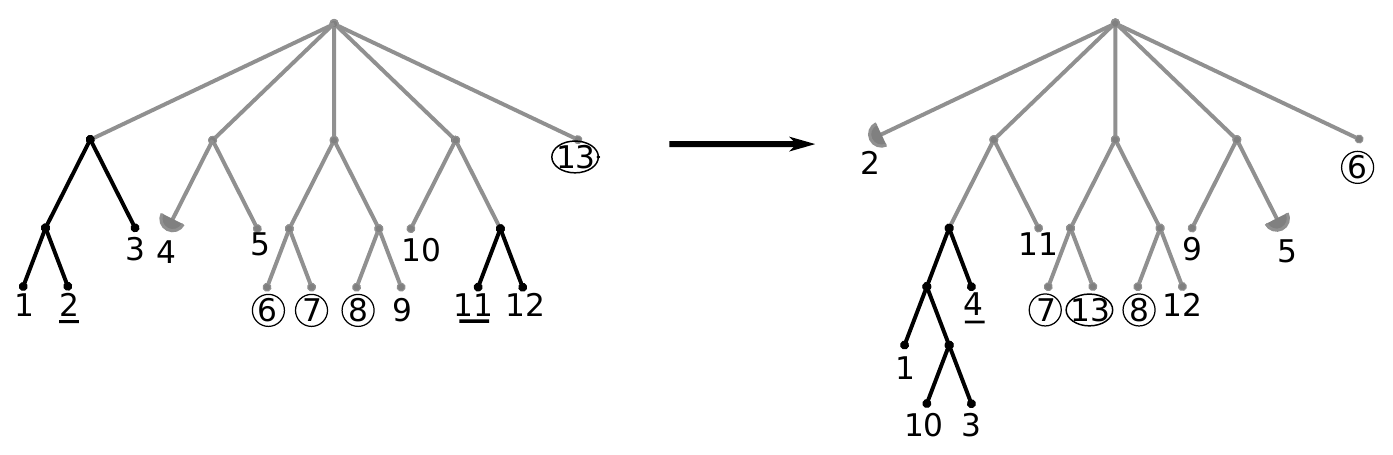}
    \caption{a revealing pair}%
\label{fig:revealing pair}
\end{figure}
\end{example}
    
\begin{remark}\label{rem:four_types_of_leafs}
    It is not hard to see that a tree pair is a revealing pair if and only if all of its chains are attractor, repeller, periodic and wandering chains. A detailed proof can be found in \cite{sd10}, Claim 5.
\end{remark}

Every tree almost automorphism defines many tree pairs.

\begin{definition}
    Let $g \in \AAutT$. Let $T_1$ and $T_2$ be complete finite subtrees of $\T$ such that there exists a forest isomorphism $\foris{\phi}{T_1}{T_2}$ representing $g$. Then we denote the restriction of $\phi$ to the leaves of $T_1$ by $\overline{g} := \varphi|_{\LL T_1} \colon \LL T_1 \to \LL T_2$, and the tree pair $\tp{g}{T_1}{T_2}$ we call a \emph{tree pair associated to $g$}.
\end{definition}

It is an easy exercise to show that $\tp{g}{T_1}{T_2}$ depends, as the notation suggests, only on $g$ and on the trees $T_1$ and $T_2$, but not on $\phi$.

\begin{remark}\label{rem:tree pairs define topology}
    Note that for every tree pair $P$ the set of tree almost automorphisms $g$ such that $P$ is a tree pair associated to $g$ is open. In fact, the collection of open sets of this form is a basis for the topology on $\AAutT$.
\end{remark}

In the other direction, given a tree pair we can associate it with an almost automorphism. However, going in this direction, more choice is required. We will, by convention, take a Higman--Thompson element.

\begin{definition}\label{def:HT_induced_by_tree_pair}
    Let $P=[\kappa,T_1,T_2]$ be a tree pair. The almost automorphism \emph{induced by} $P$ is the Higman--Thompson element represented by the unique plane order preserving forest isomorphism $\foris{\phi}{T_1}{T_2}$ such that $\phi|_{\LL T_1}=\kappa$.
\end{definition}

Let $g \in \AAutT$ and let $\foris{\phi}{T_1}{T_2}$ be a forest isomorphism representing $g$. Let $x \in \LL T_1$ and let $T$ be a complete finite subtree rooted at $x$.
It is obvious how to enlarge $T_1$ with $T$ to get a tree pair for $g$, namely simply take the tree pair $\tp{g}{T_1 \cup T}{T_2 \cup \phi(T)}$, where $\overline{g}$ is the restriction of $\phi$ to the leaves of $T_1 \cup T$.

If we consider a maximal chain $(x_0,\dots,x_n)$, it can be useful to enlarge the tree pair in such a way that a pre-determined tree is attached to $x_0$, but no components are added under $x_1,\dots,x_{n-1} \in \LL T_1 \cap \LL T_2$.
This leads us to the following notion introduced by Salazar-D\'iaz \cite{sd10}, Definition 22.

\begin{definition}\label{def:rolling}
    Let $g \in \AAutT$, let $\foris{\phi}{T_1}{T_2}$ be a representative of $g$
    and let $P := \tp{g}{T_1}{T_2}$ be a tree pair associated to $g$. Let $(x_0,\dots,x_n)$ be a maximal chain of $P$.
    \begin{enumerate}
        \item Let $T$ be a complete finite subtree of $\T$ that does not contain the root, but is rooted at $x_0$. The \emph{forward $g$-rolling of $P$ with $T$ along $(x_0,\dots,x_n)$} is the tree pair
     $\tp{g}{T_1 \cup T \cup \phi(T) \cup \dots \cup \phi^{n-1}(T)}{T_2 \cup \phi(T) \cup \dots \cup \phi^{n}(T)}$.
     \item Let $T$ be a complete finite subtree of $\T$ that does not contain the root, but is rooted at $x_n$. The \emph{backward $g$-rolling of $P$ with $T$ along $(x_0,\dots,x_n)$} is the tree pair
     $\tp{g}{T_1 \cup \phi^{-1}(T) \cup \dots \cup \phi^{-n}(T)}{T_2 \cup T \cup \phi^{-1}(T) \cup \dots \cup \phi^{-(n-1)}(T)}$.
    \end{enumerate}
\end{definition}

By convention, if we do not specify the direction of the rolling, we mean a forward rolling except in the case of a repeller chain.

\begin{example}
 Figure \ref{fig:rolling} gives an example of a backward rolling for the Higman--Thompson element $g$ induced by the tree pair $P=[\kappa,T_1,T_2]$ depicted.
 The maximal chain along which the rolling is done is $c=(\kappa^{-1}(7),7,6,5)$ expressed in labels in $T_2$, which is the same as $c=(7,6,5,\kappa(5))$ expressed in labels in $T_1$.
 The tree $T$ is the gray subtree of the first picture, which hangs at the vertex  $5 \in T_2$.
 Performing the $g$-backward rolling of $P$ with $T$ along $c$ includes gluing copies of $T$ to the leaves $5,6$ and $7$ in $T_1$, and to the leaves $5,6$ and $7$ in $T_2$. 
 
\begin{figure}
    \centering
    \subfloat[a tree pair]{\includegraphics[scale=1]{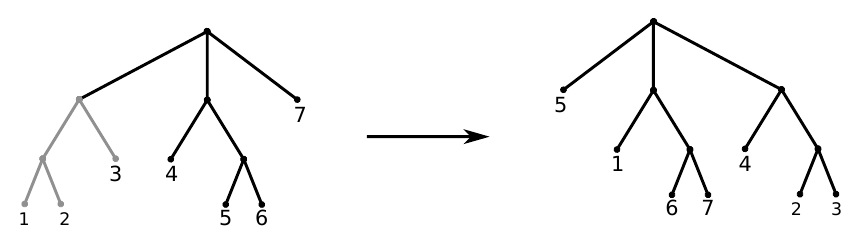}}

    \subfloat[its rolling]{\includegraphics[scale=1]{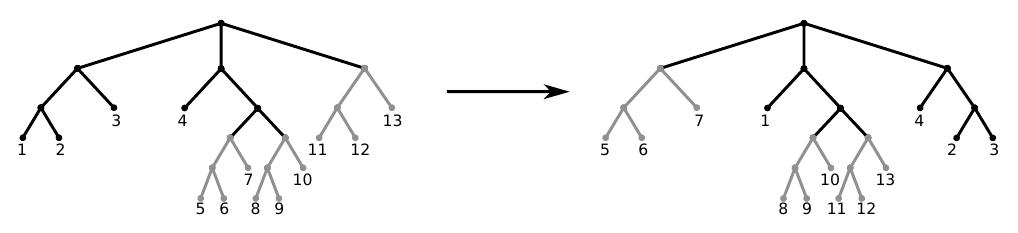}}%
    \caption{A tree pair and its rolling with a component of $T_1 \setminus T_2$.}%
\label{fig:rolling}
\end{figure}
Observe that $P$ is not a revealing pair. Indeed, $5\in T_2$ is the root of a component of $T_1\setminus T_2$, which contains no repeller.
However, the rolling of $P$ is a revealing pair, which is an illustration of the proof of Lemma \ref{lem: constructing a revealing pair}.
\end{example}

Rollings are useful tools to produce revealing pairs. For example, using the correct trees, one can produce new revealing pairs from old ones.

\begin{definition}
    Let $g \in \AAutT$ and $P=\tp{g}{T_1}{T_2}$ a revealing pair for $g$. Let $(x_0,\dots,x_n)$ be a maximal chain. A \emph{cancelling tree} for $g$ at $(x_0,\dots,x_n)$ is a tree $T$ such that the $g$-rolling of $P$ along $(x_0,\dots,x_n)$ with $T$ is again a revealing pair.
\end{definition}

The existence of cancelling trees was proven by Salazar-D\'iaz (see Definition~20 and Claim 7 in \cite{sd10}). For a  wandering chain, any tree is a cancelling tree.
For a repeller chain, an example of a cancelling tree is the component of the repeller, for an attractor chain, the component of the attractor. For a periodic chain, an example is a caret.

We now show how to use rollings to produce revealing pairs from arbitrary tree pairs. The existence of revealing pairs for Higman--Thompson elements was proved by Brin in \cite{brin04}, Argument~10.7.
However, Brin's proof is not constructive. As our procedure to classify conjugacy in $\AAutT$ requires revealing pairs for all elements of $\AAutT$, we include here a new proof, which is constructive.

\begin{lemma}[Constructing a revealing pair] \label{lem: constructing a revealing pair}
	Let $g\in \AAutT$ and let $\tp{g}{T_1}{T_2}$ be a tree pair associated to $g$. Then there exist finite complete subtrees $T^+_1$ and $T^+_2$ of $\T$ with $T^+_i\supset T_i$ such that $\tp{g}{T_1^+}{T_2^+}$ is a revealing pair associated to $g$.
\end{lemma}

\begin{proof}
    	For a tree pair $P=\tp{g}{S_1}{S_2}$, we call a component of $S_1 \setminus S_2$ a fake repelling component of $P$ if it does not contain a repeller. Similarly, a component of $S_2 \setminus S_1$ will be called a fake attracting component if it does not contain an attractor. By definition, $P$ is revealing if and only if it has no fake components. The idea of the proof is to perform rollings with fake components until no such components are left.
	    
	    \emph{Claim 1:} Let $P$ be a tree pair associated to $g$ and let $A$ be a fake attracting component. Let $x_0$ be the root of $A$ and $(x_0,\dots,x_n)$ its maximal chain. Let $Q$ be the forward $g$-rolling of $P$ with $A$ along $(x_0,\dots,x_n)$. Then, either the number of fake attracting  components in $Q$ is smaller than in $P$, or
	    $Q$ has strictly less fake attracting components than $P$ but the total number of carets involved in fake attracting components of $Q$ is the same or less than in $P$. The analog statement holds with fake repelling components.
	    
	    \emph{Claim 2:} Let $P$ be a tree pair associated to $g$ without fake attracting components. Let $B$ be a fake repelling component of $P$, let $x_n$ be the root of $B$ and let $(x_0,\dots,x_n)$ be its maximal chain. Let $Q$ be the backward $g$-rolling of $B$ along $(x_0,\dots,x_n)$.
	    Then $Q$ does not have any fake attracting components.
	    
	    The lemma clearly follows from these two claims. Indeed, given a tree pair $P$ for $g$, we perform $g$-forward rollings with fake repelling components until none are left, by Claim $1$ this is a finite process. Then, we perform $g$-backward rollings with fake attracting components until none are left. By Claim $2$ we will not create any new fake repelling components, and by Claim $1$ it is again a finite process.
	    
	    \emph{Proof of Claim 1:} It suffices to prove the statement for fake attracting components. The case of fake repelling components works completely analogously.
	    Let $P=:\tp{g}{S_1}{S_2}$ be a tree pair associated to $g$ and let  $\foris{\phi}{S_1}{S_2}$ be the corresponding representative.
	    Let $A,(x_0,\dots,x_n)$ and $Q:= \tp{g}{S_1'}{S_2'}$ be as in the claim.
    Observe that all components of $S_2 \setminus S_1$ except $A$ remain untouched by the rolling. As regards $A$, it will not appear as a fake repelling component of $Q$, because it appears in $S_1'$ as well. However, we may have created new fake attracting components while performing the rolling.
    	The glued copies of $A$ rooted in the neutral leaves $x_1,\dots,x_{n-1}$ were added in both $S_1$ and $S_2$ and so they have no contribution to the set of components of $S_2' \setminus S_1'$. It remains to look at the tree $\phi^n(A)$ glued at $x_n \in S_2$.	
        Because the chain is maximal, the vertex $x_n$ is not a leaf of $S_1$. Hence, $x_n$ either does not belong to $S_1$, or it is an inner vertex of $S_1$.
	    In the first case $\phi^n(A)$ was glued to a component of $S_2 \setminus S_1$ not equal to $A$, and it has no influence on whether it was a fake component or not, since it was not glued to a vertex in the $\phi$-orbit of the root of that component. Hence in this case no new fake repelling components were added, and so the number of fake components strictly decreased. The number of carets involved in fake attracting components did not increase because only a copy of $A$ was added to a component of $S_2 \setminus S_1$.
	    In the second case, since $x_n$ is an inner vertex of $S_1$, possible new components in $S_2' \setminus S_1'$ have in total less carets than $A$.
	
	\emph{Proof of Claim 2:}
	It is only possible that the $g$-rolling produces fake attracting components if $x_0$ is an inner leaf of $T_2$.
	In this case $x_0$ is a root of a component $A$ of $T_2 \setminus T_1$.
	But because $P$ does not have any fake attracting components, $x_0$ cannot be in the $\phi^{-1}$-orbit of an attractor. So we get that $A$ was a fake attracting component, contradicting the assumption that there are none of those.
\end{proof}

\subsection{Strand diagrams}\label{strand diagrams}

Belk and Matucci used strand diagrams to solve the conjugacy problem in Thompson's group $V$.
We follow their approach here and refer to their article \cite{bema14} for more information and background. Like them we use the slightly unusual notion of a "topological graph": In a directed graph we allow connected components that do not have any vertices at all and call them "free loops".

\begin{definition}
    Let $D$ be a directed graph. A \emph{split} in $D$ is a vertex with exactly one incoming edge and at least two outgoing edges. A \emph{merge} in $D$ is a vertex with exactly one outgoing edge and at least two incoming edges.
\end{definition}

\begin{definition}
	A \emph{closed abstract strand diagram of degree $d$} consists of the following:
	\begin{itemize}
		\item a finite directed graph $D$ such that every vertex is a split with $d$ outgoing edges, or a merge with $d$ incoming edges;
		\item a map $r$, called \emph{rotation system}, defined on the set of vertices of $D$, that associates to every split a total order on its outgoing edges, and to every merge a total order on its incoming edges;
		\item a cohomology class, called \emph{cutting class}, $c \in H^1(D,\mathbb{Z})$.
	\end{itemize}
For convenience, throughout the paper we abbreviate the term \emph{closed abstract strand diagram} as \emph{BM-diagram}.
\end{definition}

Recall that a cohomology class representative $\gamma \colon \Edge(D) \to \mathbb{Z}$ is a coboundary if and only if it evaluates $0$ along every cycle. This cycle need not be directed, but if it travels along an edge $e$ in its opposite direction, we have to count $-\gamma(e)$. In particular, the total value of a cycle is independent of the representative.

\begin{remark}
    Recall the classical fact that there is a natural bijection between $H^1(D,\mathbb{Z})$ and homotopy classes of continuous maps of a geometric realization of $D$ to $\mathbb{R}^2 \setminus \{0\}$. The reason is that the punctured plane is an Eilenberg--MacLane space of type $K(\mathbb{Z},1)$. We refer to \cite{hat01}, Introduction to Chapter 3, "The idea of cohomology" for an explanation how this works. This allows us to do drawings of BM-diagrams that have all the information about rotation systems and cutting classes.
\end{remark}

\begin{example}
    Figure \ref{fig:BMdiag} shows an example of a BM-diagram. First we give it with a cohomology class representative, then as homotopy class of an embedding into the punctured plane. Note that edges with a positive label wind as often around the central hole as the label says.
 \begin{figure}
  \centering
    \subfloat[with cohomology class representative]{\includegraphics[scale=1]{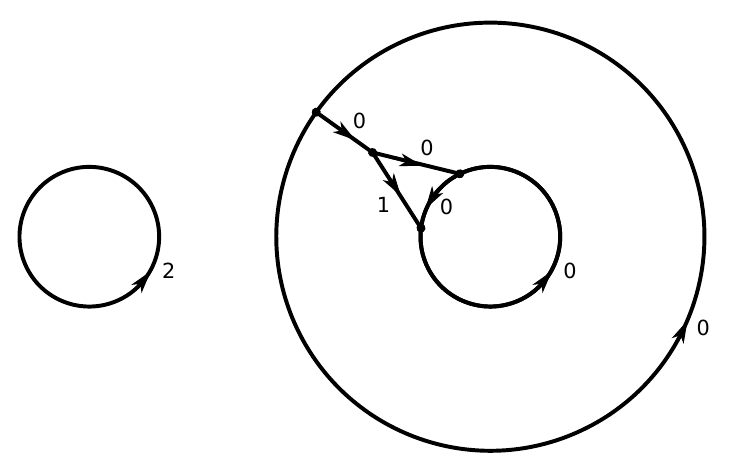}}%
    \qquad
    \subfloat[homotopy type of embedding]{\includegraphics[scale=1]{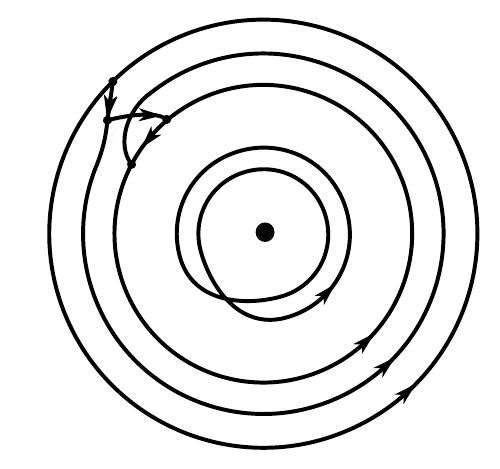}}%
  \caption{a BM-diagram.}
   \label{fig:BMdiag}
 \end{figure}
\end{example}

Let $(D,r,c)$ be a BM-diagram and $D'$ a directed graph isomorphic to $D$. A graph isomorphism $F \colon D \to D'$ clearly induces a rotation system $r_F$ and a cutting class $c_F$ on $D'$.

\begin{definition}
    Let $(D,r,c)$ and $(D',r',c')$ be two BM-diagrams. An \emph{isomorphism} between them is a graph isomorphism $F \colon D \to D'$ such that $r' = r_F$ and $c' = c_F$.
\end{definition}

Belk and Matucci defined several operations on BM-diagrams, called Type I, Type~II and Type III reductions.
The reductions induce an equivalence relation on diagrams, namely: two diagrams are equivalent if they can be reduced to the same diagram.
In the present work we will not need the third kind, but we introduce it for completeness. 
Also, we introduce a more general version of Type~I reductions that we call Type I*.

\begin{definition}
    Let $(D,r,c)$ be a BM-diagram and let $\gamma \colon \Edge(D) \to \mathbb{Z}$ be a representative for $c$.
    
	A \emph{Type I* reduction} is the following operation on a BM-diagram.
	Assume there are edges $e_1,\dots,e_d$ such that $o(e_1)=\dots=o(e_d) =: s$ is a split and $t(e_1)=\dots=t(e_d) =: m$ is a merge. Assume further that for one (and hence all) representatives $\gamma$ of the cohomology class we have $\gamma(e_i)=\gamma(e_j)$ for all $1 \leq i,j \leq d$. Then we delete the edges $e_1,\dots,e_d$ and make a new edge $e$ by melting together the incoming edge $e_s$ of $s$ and the outgoing edge $e_m$ of $m$.
	The rotation system of the new diagram is obvious, $e$ simply takes the place of $e_m$ and $e_s$ if they were part of a total order.
	The new cutting class is obtained by setting $\gamma(e) := \gamma(e_s)+\gamma(e_m)+\gamma(e_1)$ and leaving $\gamma$ unchanged in the rest of the diagram.
	
	A \emph{Type I reduction} is a Type I* reduction in the case where the order of the outgoing edges from the split is the same as the order they have when coming in to the merge. That is, $r(s) = r(m)$ as functions $\{e_1,\dots,e_d\} \to \{1,\dots,d\}$.

	A \emph{Type II reduction} is the following operation on a BM-diagram.
	Let $e$ be an edge in $D$ such that $o(e):=m$ is a merge and $t(e):=s$ is a split.
	First we erase $e$ including its endpoints from the diagram. Then for $i=1,\dots,d$ we create a new edge $e_i$ by melting together the $i$th incoming edge $e_i^m$ of $m$ with the $i$th outgoing edge $e_i^s$ of $s$. Note that it could happen that $e_i^m=e_i^s$, in which case we get that $e_i$ is a free loop.
	The new rotation system is obvious: The new edge $e_i$ simply takes the place of $e_i^m$ or $e_i^s$ in any total order they were part of.
	The cutting class is given by assigning to the new edges the value $\gamma(e_i) := \gamma(e)+\gamma(e^m_i)+\gamma(e^s_i)$ and leaving $\gamma$ unchanged on the rest of the diagram.
    
	A \emph{Type III reduction} is the following operation on a BM-diagram. If there are $d$ free loops $e_1,\dots,e_d$ such that $\gamma(e_1)=\dots=\gamma(e_d)$, then we erase $e_2,\dots,e_{d}$ and restrict $\gamma$ in the obvious way. Since there are no splits or merges involved in this operation, there is nothing to say about the rotation system.
\end{definition}

The different reduction Types are illustrated in Fig. \ref{fig:reductions}. To see Type II and Type I* illustrated on a closed loop, consult Fig. \ref{fig:problemwithIstar}.

 \begin{figure}
  \centering
    \subfloat[Type I]{\includegraphics[scale=1]{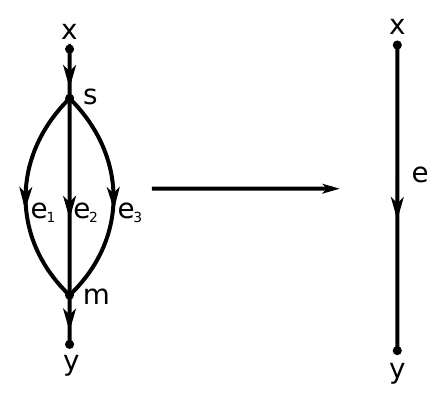}}%
    \qquad
    \subfloat[Type I*]{\includegraphics[scale=1]{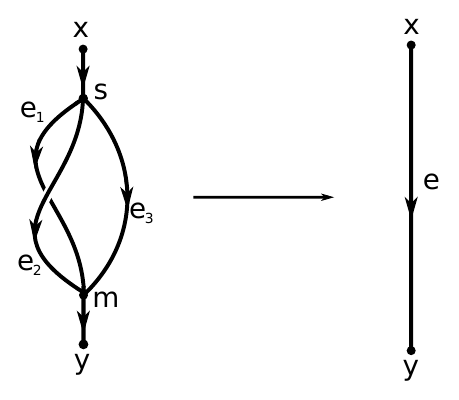}}%
    
   \subfloat[Type II]{\includegraphics[scale=1]{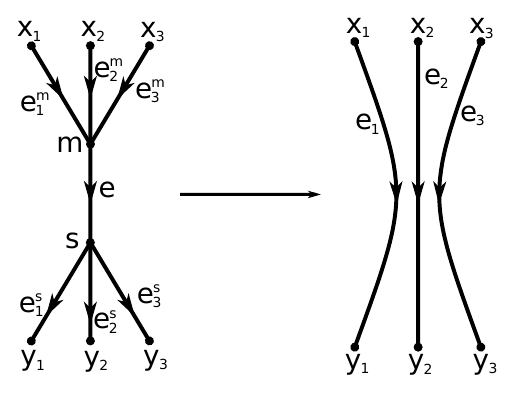}}%
    \qquad
    \subfloat[Type III]{\includegraphics[scale=0.9]{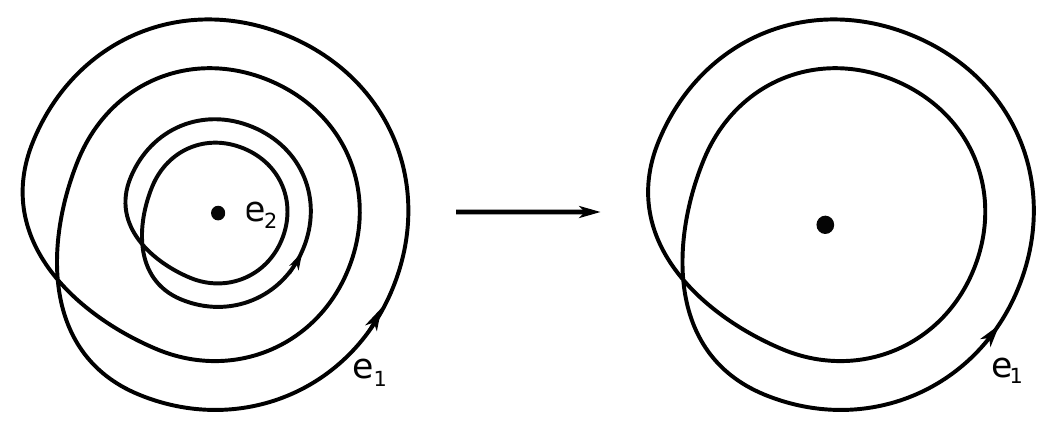}}%
  \caption{Reductions of Types I, I* and II for $d=3$, and of Type III for $d=2$.}
   \label{fig:reductions}
 \end{figure}

We now introduce three different notions of reduced BM-diagrams.

\begin{definition}
        A BM-diagram is called \emph{II-reduced} if no Type II reduction can be done on it, i.e. if there is no edge $e$ that is the outgoing edge of a merge and the incoming edge of a split.
        
	    A BM-diagram is called \emph{reduced} if no Type I, Type II or Type III reduction can be done on it.
	    
	    A BM-diagram is called \emph{*-reduced} if no Type I*, Type II or Type III reduction can be done on it.
\end{definition}

Clearly, *-reduced implies reduced.
Regarding the structure of reduced BM-diagrams, Belk and Matucci showed the following.

\begin{proposition}[\cite{bema14}, Proposition 4.1]\label{prop: loops in reduced diagram}
Let $(D,c,r)$ be a reduced BM-diagram. Let $L$ be a directed loop in $D$.
Then $L$ satisfies one of the following.
\begin{enumerate}
    \item Every vertex in $L$ is a split.
    \item Every vertex in $L$ is a merge.
    \item $L$ is a free loop, i.e., it contains no vertices.
\end{enumerate}
Moreover, all directed loops in $D$ are disjoint.
\end{proposition}

The reason why we do not bother about Type III reductions is that they only deal with free loops. Free loops represent periodic behaviour of tree almost automorphisms, and the periodic behaviour in the group $\AAutT$ is much more complicated than the one in $\V$, and so these reductions do not help to analyze the $\AAutT$ case.

Belk and Matucci showed that the reduction process, using reductions of Types I, II and III is  well-defined, in the sense that the reduced form of a diagram does not depend on the order of reductions (Proposition 2.3 in \cite{bema14}).
It is interesting to note the following.

\begin{lemma} \label{lem:reductions_do_not_destroy_reducedness}
	Let $D$ be a II-reduced BM-diagram. Suppose we perform a Type I* reduction on $D$. Then the resulting diagram is still II-reduced.
	
	It follows that the following process, done on a given BM-diagram, results in a (*-)reduced diagram. First perform Type II reductions until the diagram is II-reduced, then perform on it Type I(*) reductions until it is not possible anymore, and lastly perform Type III reductions until none are possible anymore.
\end{lemma}

\begin{proof}
	For the first part, let $s$ be the split and $m$ the merge that vanished in the Type I* reduction. Let $e_s$ be the edge ending at $s$ and $e_m$ be the edge starting at $m$, and denote by $s=o(e_s)$ and $y=t(e_m)$. 
	Note that $x$ is a split and $y$ a merge because $D$ is II-reduced.
	This means that the new edge connecting $x$ to $y$, which we have after the Type I* reduction, is not subject to Type II reduction. But since the rest of the diagram is unchanged, this implies the claim.
	
	The second part of the lemma follows directly from the first. 
\end{proof}

The next few paragraphs deal with the question when an isomorphism between BM-diagrams survives a Type II reduction. This will play a crucial role in the proof of Lemma \ref{lem: fixator elements change only the rotation system}.

\begin{definition}
	Let $(D,c,r)$ a BM-diagram. A sub-diagram of $D$ is called an \emph{hourglass} if it consists of the following:
	\begin{itemize}
		\item a complete tree $T_1$ all of whose inner vertices are merges. In particular, all of its maximal directed paths end in a vertex $r_1$.
		\item a complete tree $T_2$ that is the mirrored copy of $T_1$ in the sense that the directions of all edges are reversed, but the rotation system is unchanged.
		In particular, all inner vertices of $T_2$ are splits, and all maximal directed paths start in a vertex $r_2$.
		\item a directed edge going from $r_1$ to $r_2$
	\end{itemize}
	Two vertices $x_1\in T_1,x_2\in T_2$ in an hourglass are called \emph{correlated}, if $x_2$ is the image of $x_1$ under the direction-reversing identification of $T_1$ with $T_2$. Note that in particular it follows that $x_1$ is merge and $x_2$ is a split.
\end{definition}

The simplest example of an hourglass is a merge that is followed by a split, which is exactly the situation when we can perform a Type II reduction.
The point of an hourglass is that we can make it vanish by repeatedly performing Type II reductions.

\begin{definition}
    Let $H$ be an hourglass in a BM-diagram with merge tree $T_1$ and split tree $T_2$. A \emph{Type II reduction of $H$} is the following operation. First, delete the interior of $H$. Then melt together each edge ending at a leaf of $T_1$ with the edge starting at the correlated leaf of $T_2$. Equivalently, perform repeatedly Type II reductions on all edges in $H$, until all its interior is gone.
\end{definition}

\begin{figure}
    \centering
    \includegraphics[width=\textwidth]{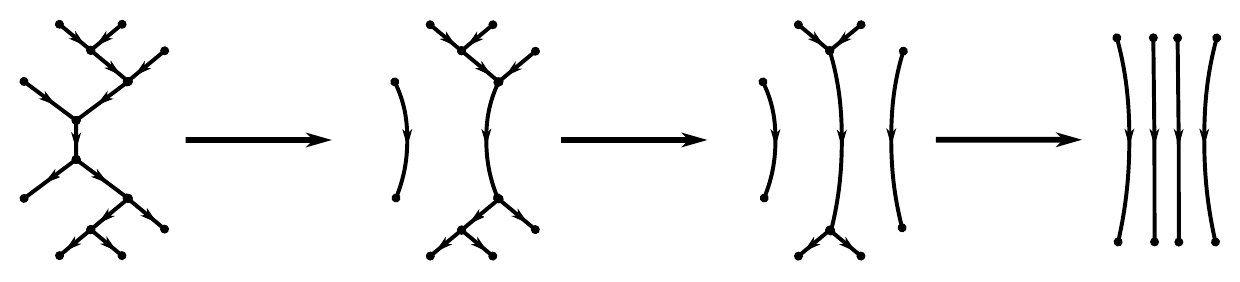}
    \caption{Type II reduction of an hourglass.}
    \label{fig:reduction of an hourglass}
\end{figure}

\begin{definition}
    Two BM-diagrams $(D,c,r)$ and $(D',c',r')$ are said to be \emph{isomorphic up to rotation} if there exists a graph isomorphism $F\colon D \to D'$ such that $c'=c_F$.
\end{definition}

That is, the two diagrams are isomorphic as directed graphs with a cohomology class, but the isomorphism between them does not necessarily respect the rotation system. 

Being isomorphic up to rotation is not preserved under Type II reductions in general.
The problem is that if a Type II reduction melts together two edges $e,f$ in $D$, there is no reason why $F(e)$ and $F(f)$ would be melted together as well, see Fig. \ref{fig:isom not respecting an hourgalss}.

\begin{figure}
    \centering
     \subfloat{\includegraphics[scale=0.8]{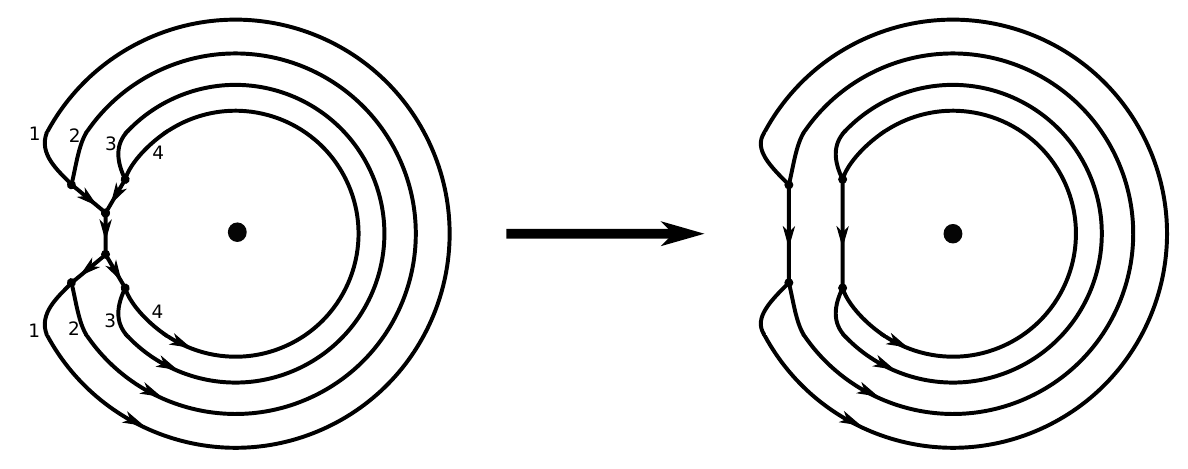}}%
    \qquad
    \subfloat{\includegraphics[scale=1]{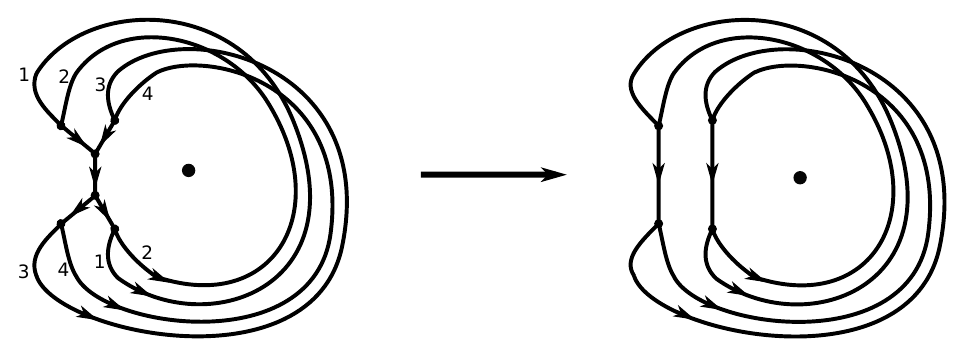}}%
    \caption{The isomorphism does not respect hourglasses, so the reduced diagrams are not isomorphic up to rotation.}
    \label{fig:isom not respecting an hourgalss}
\end{figure}
It is too strong to ask that $F$ does not do anything to the rotation system at the different endpoints of an edge connecting a merge to a split; it suffices to require that $F$ messes up both total orders by the same permutation.

\begin{definition}
    Let $(D,c,r)$ and $(D',c',r')$ be two BM-diagrams of degree $d$ and let $F \colon D \to D'$ be a graph isomorphism. Let $H \subset D$ be an hourglass. Then we say that $F$ \emph{respects} $H$ if
    for all correlated inner vertices $x,y$ of $H$ there exists a $\sigma \in \operatorname{Sym}(d)$ such that $r_F(F(x)) = r'(F(x)) \circ \sigma$ and $r_F(F(y)) = r'(F(y)) \circ \sigma$.
\end{definition}

Note that if $F$ respects $H$ then $F(H)$ is an hourglass in $D'$ and $x,y$ are correlated vertices if and only if $F(x)$ and $F(y)$ are.

\begin{lemma}\label{lem: diff_in_rot_system_preserved_under_II-reduction}
	Let $(D,c,r)$ and $(D',c',r')$ be two BM-diagrams that are isomorphic up to rotation via a graph isomorphism $F \colon D \to D'$.
	Let $H$ be an hourglass in $D$ and assume that $F$ respects $H$. Then, after performing the Type II reduction on $H$ and $F(H)$, the diagrams are still isomorphic up to rotation via an isomorphism induced by $F$.
\end{lemma}

\begin{proof}
	Suppose first that a Type II reduction is done on an edge $e=(m,s)$. Respecting the rotation system at $(m,s)$ means that after the Type II reduction, if the edges $e^i_m$ (ending at $m$) and $e^i_s$ (starting at $s$) melted to one edge, $e^i$, then also $F(e^i_m)$ and $F(e^i_s)$ melted to one edge, $e^{i'}$. Abusing notation we denote by $F$ also the new isomorphism, then $F(e^i)=e^{i'}$.
	
	Since a Type II reduction of hourglass can be done by successively Type II reducing along single edges, the statement now follows by induction.
\end{proof}

\begin{example}
Figure \ref{fig:isom not respecting an hourgalss} illustrates an isomorphism up to rotation that does not respect hourglasses. As a consequence, the reduced diagrams are not isomorphic up to rotation. Indeed, they have differently many connected components.
\end{example}

\begin{corollary}\label{cor: diff_in_rot_system_preserved_under_reduction}
	Let $(D,c,r)$ and $(D',c',r')$ be two BM-diagrams that are isomorphic up to rotation via a graph isomorphism $F \colon D \to D'$, and suppose $F$ respects all hourglasses in $D$. Assume that after performing the Type II reductions on all these hourglasses the diagrams are II-reduced. Then the *-reductions of $D$ and $D'$ are isomorphic up to rotations.
\end{corollary}

\begin{proof}
	Let $\hat{D},\hat{D'}$ denote the Type II reductions of $D,D'$ respectively. By Lemma \ref{lem: diff_in_rot_system_preserved_under_II-reduction}, $\hat{D}$ and $\hat{D'}$ are isomorphic up to rotation. Type I* reductions do not depend on the rotation system, so they do not affect it. It follows that the isomorphism between $\hat{D}$ and $\hat{D'}$ descents to an isomorphism between their *-reductions, preserving the cutting classes. 
\end{proof}

\subsubsection{From tree pairs to strand diagrams and back}\label{from tree pairs to strand diagrams}

Every tree pair gives rise to a BM-diagram.
The plane order on the trees in the pair, inherited from the plane order on $\T$, will induce the rotation system.

\begin{definition}\label{def:basic BM-diagram}
	Let $P=[\kappa,T_1,T_2]$ be a tree pair.
	 The \emph{basic BM-diagram of $P$} is the BM-diagram constructed as follows:
	\begin{enumerate}
		\item Draw a copy of $T_1$ and direct all edges to point away from the root $r_1$. Keep the plane order of the outgoing edges in every vertex.
		\item Draw a copy of $T_2$ and direct all edges to point toward the root $r_2$. Keep the order of the incoming edges in every vertex.
		\item Identify each leaf $x$ of $T_1$ with the leaf $\kappa(x)$ of $T_2$. In particular, the edge ending at $x$ and the edge starting at $\kappa(x)$ merge to a single edge.
		\item Put an edge $e$ with $o(e)=r_2$ and $t(e)=r_1$.
		\item Define a cutting class $[\gamma]$ of via $\gamma(e):=1$ and $\gamma(e')=0$ for all edges $e'\neq e$.
		\item Note that $e$ together with the two copies of $T_1 \cap T_2$ form an hourglass, and a vertex $v \in T_1 \cap T_2$ viewed as vertex of $T_1$ is simply correlated to itself viewed as vertex of $T_2$. Do a Type II reduction on this hourglass.
	\end{enumerate}
		\begin{figure}[H]
    \centering
    \subfloat[a revealing pair for an element]{\includegraphics[scale=0.8]{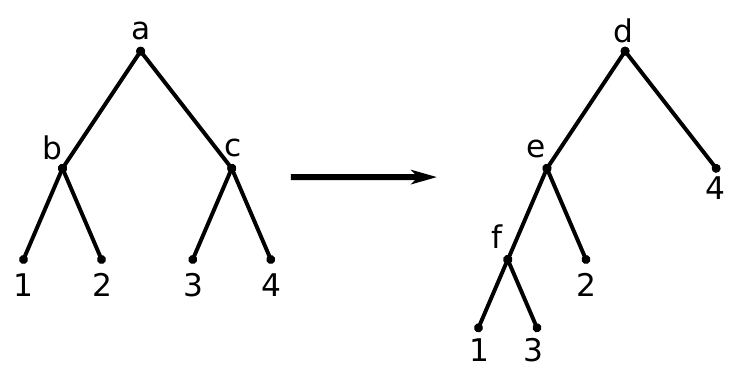}}%
  
    \subfloat[forming its basic BM-diagram]{\includegraphics[scale=0.6]{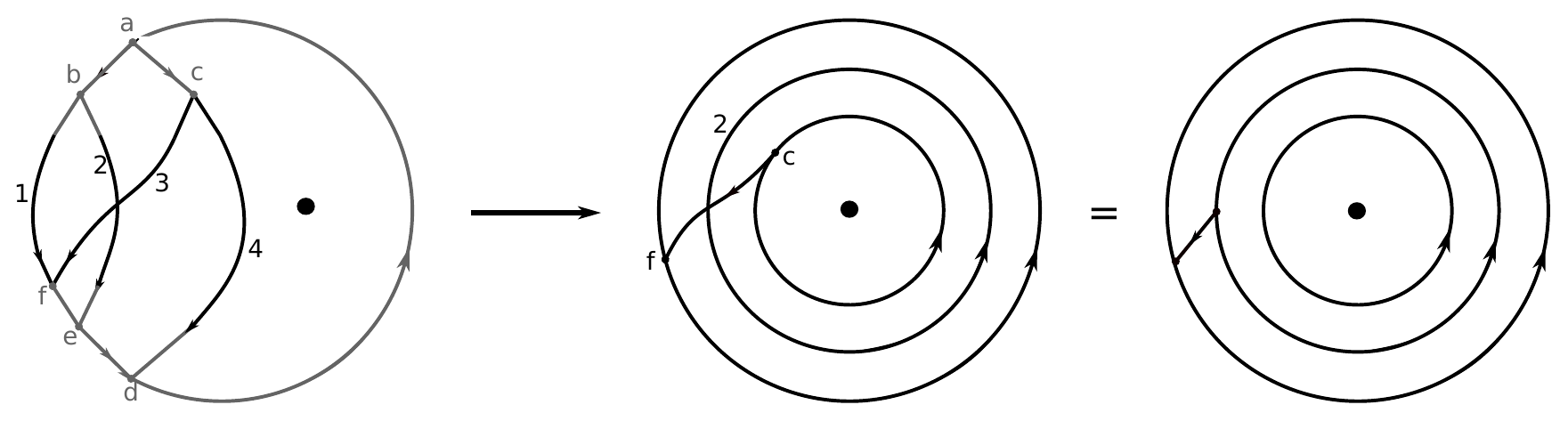}}%
    \caption{The basic BM-diagram of a tree pair.}%
\label{fig:basic BM-diagram}%
\end{figure}
\end{definition}

\begin{remark}
    The basic BM-diagram of a tree pair of $\Tdk$ is indeed a BM-diagram of degree $d$. This is because the hourglass being reduced in the last step always contains the root and the $k$ edges adjacent to it.
\end{remark}

An example for a tree pair and its basic BM-diagram is shown in Fig. \ref{fig:basic BM-diagram}. The hourglass is drawn with gray edges.

Basic BM-diagrams behave nicely with respect to revealing pairs.

\begin{lemma}\label{lem:basic_BM_diagram_of_revealing_pair_II-reduced}
    The basic BM-diagram of a revealing pair is II-reduced.
\end{lemma}

\begin{proof}
    Let $D$ be the basic BM-diagram of the revealing pair $P=[\kappa,T_1,T_2]$.
    Note that all the vertices of $D$ can be identified with roots and inner vertices of components of $T_1$ and $T_2$, where the vertices from $T_1$ stay splits and the vertices of $T_2$ stay merges.
    
    Let now $e$ be an edge in $D$ that starts in a merge $m$. We have to show that the end of $e$ is a merge as well.
    If $m$ was an inner vertex of a component of $T_2 \setminus T_1$, it is followed by another merge. We can therefore assume that $m$ is the root of an attracting component $A$ in $T_2$, and therefore $m$ was in the hourglass that got reduced.
    So it had a correlated vertex in $T_1$ before the hourglass reduction, which was clearly the vertex $m$ in $T_1$. But $m \in T_1$ was connected to $\kappa(m) \in T_2$, the correlated vertex of which was $\kappa (m) \in T_1$, and so on. Since $P$ is a revealing pair, for some $n$ the vertex $\kappa^n(m)$ is a leaf of $A$. It follows that $e$ is an incoming edge of a merge in $A$, as we wanted.
\end{proof}

Belk and Matucci introduced BM-diagrams in order to classify conjugacy classes in Thompson's group $V=V_{2,2}$. They proved the following theorem. There is nothing special about $V_{2,2}$, the proofs work for all $V_{d,k}$.

\begin{theorem}[\cite{bema14}, Proposition 2.3, Theorem 2.15] \label{thm:belkmatucci}
	Let $v,w \in V$.
	Let $P=\tp{v}{T_1}{T_2}$ and $Q=\tp{w}{T_3}{T_4}$ be tree pairs for $v$ and $w$ and form their basic BM-diagrams. Perform Type I, II and III reductions on them until they are reduced. Let $(D,r,c)$ and $(D',r',c')$ be these reduced diagrams.
	\begin{enumerate}
	    \item The reduced diagrams $(D,r,c)$ and $(D',r',c')$ depend only on $v$ and $w$, but not on $P$ and $Q$ or on the order of reductions.
	    \item The elements $v$ and $w$ are conjugate in $V$ if and only if $(D,r,c)$ and $(D',r',c')$ are isomorphic.
	\end{enumerate}
\end{theorem}

\begin{remark}
Type I* reduction is problematic in this context, as illustrated in Fig. \ref{fig:problemwithIstar}. The BM-diagram on the left corresponds to an element of $V$. Allowing I* reductions, this diagram can be reduced into BM-diagrams of non conjugate elements: the right image corresponds to an element of order 2 in $V$, while the bottom one to the identity.
\end{remark}

\begin{figure}
    \centering
    \includegraphics[width=70mm]{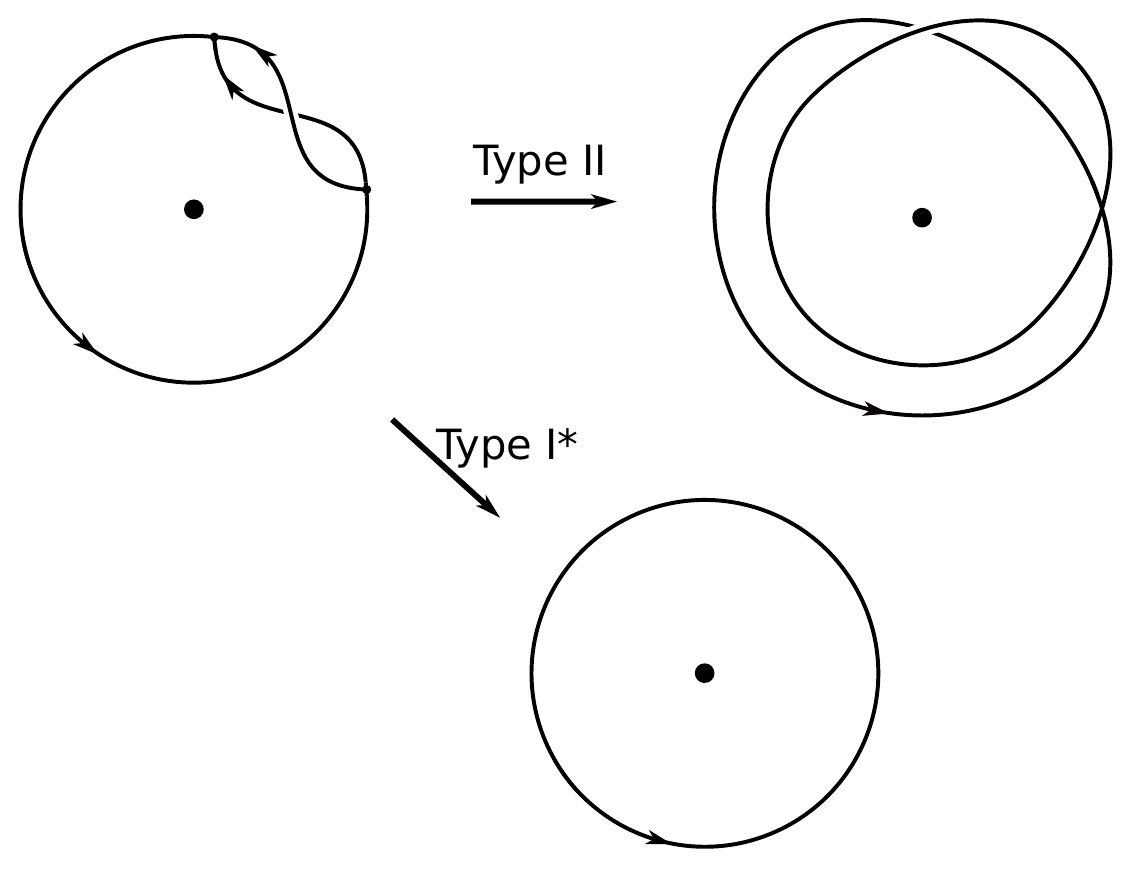}
    \caption{problematic * reductions of an element in $V$.}
    \label{fig:problemwithIstar}
\end{figure}

In Definition \ref{def:basic BM-diagram}, we saw that any tree pair gives rise to a BM-diagram. On the other direction, we now identify which BM-diagrams come from tree pairs.

\begin{definition}
    Let $D$ be a BM-diagram of degree $d$.
	A cutting class of $D$ is called \emph{$k$-admissible} if it has a representative that takes only non-negative values, gives a positive value to every directed cycle, and the sum of the values of all edges is congruent to $k \mod{d-1}$. Such a representative will be called $k$-\emph{admissible}.
\end{definition}

We remark that $k$ in the definition will always be the valency of the root of $\T=\Tdk$.
Note that for an element in $\AAutTdk$, the cutting class of the BM-diagram constructed in Definition \ref{def:basic BM-diagram} is $k$-admissible. Moreover, $k$-admissibility is preserved under reductions.
However, not all representatives of $k$-admissible cutting classes are $k$-admissible.
To construct a tree pair out of a reduced BM-diagram with $k$-admissible cutting class, we have to modify the $k$-admissible representative to a specific form.

\begin{lemma} \label{lem:cutclassrep}
	Let $c$ be a $k$-admissible cutting class on a reduced BM-diagram. Then, $c$ has a $k$-admissible representative $\gamma$ satisfying the following.
	\begin{enumerate}
	    \item For each directed loop, there is exactly one edge on which $\gamma$ is non-zero.
	    \item Outside of directed loops, $\gamma$ is non-zero at most on edges that do not connect a split to a merge.
	    \item If the diagram has vertices, then the sum of all values of $\gamma$ is at least $k$.
	\end{enumerate}
\end{lemma}

\begin{proof}
	Observe that the following elements define trivial cohomology classes. First, let $e$ followed by $e_1,\dots,e_d$ be a split. Then, a function that maps $e$ to $a$ and $e_1,\dots,e_d$ to $-a$ and is zero everywhere else is a coboundary, because clearly it evaluates zero along every directed loop.
	Moreover, the sum of its values on all edges is $a - da$, in particular it is divisible by $d-1$. The analog statement holds for merges.
	
	\begin{figure}[H]
	    \centering
	    \includegraphics{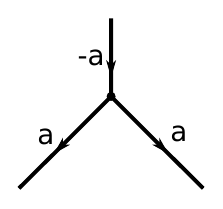}
	    \caption{Part of a representative of the trivial cohomology class.}
	    \label{fig:modyfycohclass}
	\end{figure}
	
	Let $\gamma'$ be an admissible representative of $c$.
	Recall from Proposition \ref{prop: loops in reduced diagram} that a directed loop in a reduced BM-diagram has only splits, has only merges, or has no vertices at all. Note that all split and merge loops are disjoint from one another.
	Using the above  observation, we can first modify $\gamma'$ such that for every split and merge loop there is just one edge with non-zero value.
	The procedure is illustrated in Fig. \ref{fig:modyfycohclassalongcycle}.
	Clearly, this modifications do not destroy admissibility.
	
	\begin{figure}[H]
	    \centering
	    \includegraphics[width=\textwidth]{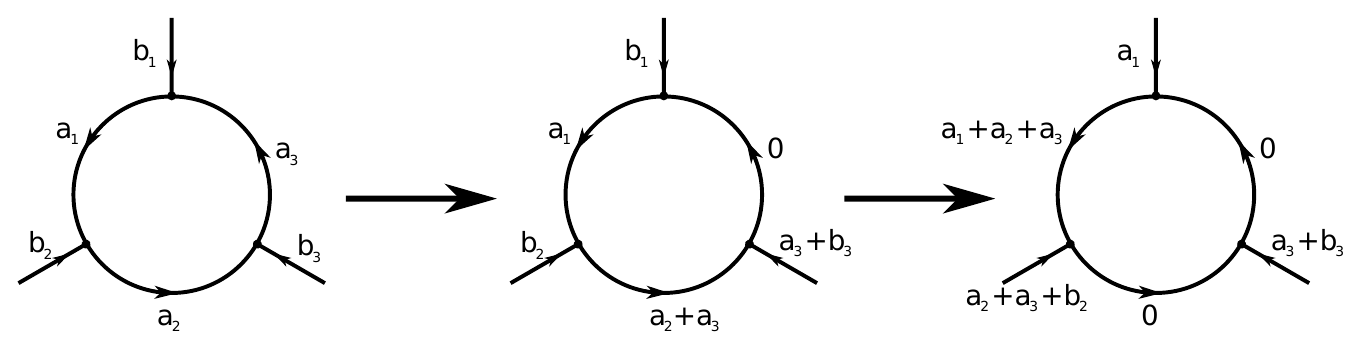}
	    \caption{Modifying the cohomology class representative along a loop.}
	    \label{fig:modyfycohclassalongcycle}
	    	\end{figure}	
	Then we can modify it further such that, outside of the split and merge loops, the incoming edge for every split and the outgoing edge for every merge have value zero, as illustrated in Fig. \ref{fig:modyfycohclasssplit}.
	
	\begin{figure}[H]
	    \centering
	    \includegraphics{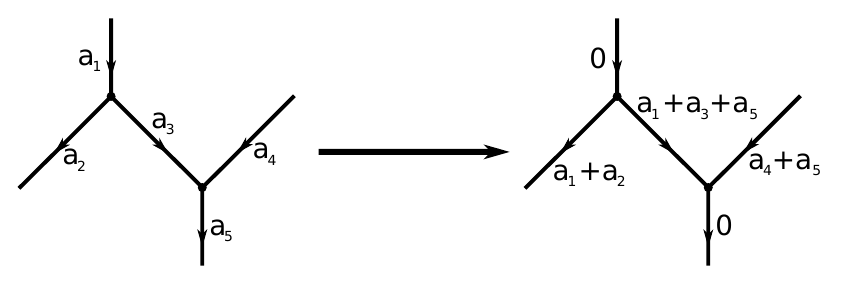}
	    \caption{Modifying the cohomology class representative between splits and merges.}
	    \label{fig:modyfycohclasssplit}
	\end{figure}

	We are left with modifying $\gamma'$ such that the total sum of all values on all edges is at least $k$. Note that if the BM-diagram has only free loops the value of $\gamma$ on each loop is completely determined by $c$.
	
	Hence we can assume that the diagram has at least one merge.
	Note that for every edge $f$ ending in a merge, there is a unique directed, semi-infinite path starting with $f$, and this path eventually winds around a merge loop indefinitely.
	Choose a merge loop $M$ in the reduced BM-diagram. Let $f_1,\dots,f_m$ be all the edges connecting a split to a merge such that this unique directed, bi-infinite path starting at $f_i$ eventually winds around $M$.
	Note that removing $f_1,\dots,f_m$ would split the connected component containing $M$ into two directed components: the one containing $M$ and the rest.
	Therefore, every undirected loop containing one of the $f_i$'s has to contain evenly many of them, and it passes through the $f_i$'s alternatingly in positive and negative direction.
	This implies that adding the same value to $\gamma(f_1),\dots,\gamma(f_m)$ does not change the cohomology class. Hence we can add a sufficiently high multiple of $d-1$ to $\gamma(f_1),\dots,\gamma(f_m)$, without destroying $k$-admissibility, such that the sum of all values of $\gamma'$ is at least $k$.
\end{proof}

The proof of the next proposition explains how to construct a revealing pair out of a reduced BM-diagram with $k$-admissible cutting class. 

\begin{proposition}
\label{prop: a reduced BM-diagram comes from a revealing pair}
Let $(D,r,c)$ be a BM-diagram of degree $d$ with a $k$-admissible cutting class.
If $D$ consists only of free loops we assume that the total value of $c$ on $D$ is at least $k$, otherwise we assume that $D$ is reduced.
Then, there exists a revealing pair $P=[\kappa,T_1,T_2]$ with $T_1,T_2 \subset \Tdk$ such that $(D,r,c)$ is the basic BM-diagram of $P$.
\end{proposition}

\begin{proof}
Fix a representative $\gamma$ of the cutting class as in Lemma \ref{lem:cutclassrep}.
	\begin{figure}[H]
	    \centering
	    \includegraphics{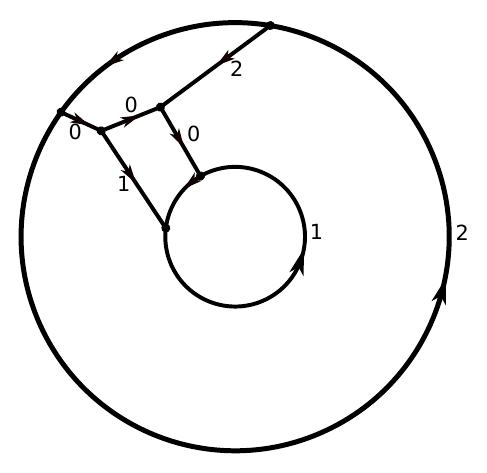}
	    \caption{A reduced BM-diagram with an admissible cutting class representative.}
	    \label{fig:reduced BM-diagram with admissible representative}
	\end{figure}
	
Cut every edge $e$ exactly $\gamma(e)$-many times.
Denote the cut points in $D$ by $p_1,\dots,p_{n}$.
For every cut point $p$ let $p^-$ and $p^+$ denote the copies of $p$ in the new diagram, such that $p^-$ is always the origin and $p^+$ the terminus of an edge. We denote this new diagram by $D'$.
	\begin{figure}[H]
	    \centering
	    \includegraphics{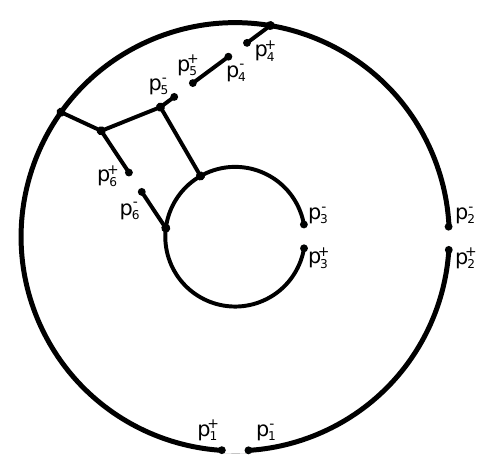}
	    \caption{The diagram $D'$.}
	    \label{fig:the diagram D'}
	\end{figure}
Let $T\subset \T$ be a finite complete tree with $n$ leaves. Note that such a tree exists because of the possible values $n$ can attain. Denote the leaves of $T$ by $p_1,\dots,p_n$.
	\begin{figure}[H]
	    \centering
	    \includegraphics{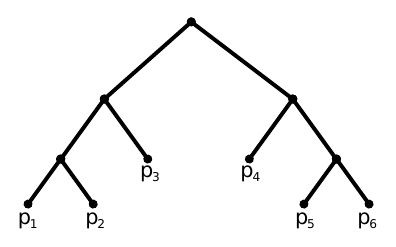}
	    \caption{The tree $T$.}
	    \label{fig:the tree T}
	\end{figure}
Let $T^-$ be a copy of $T$ in which all edges are directed away from the root. Similarly let $T^+$ be a copy of $T$ in which all edges are directed towards the root.

Glue $D'$ to $T^-$ and $T^+$ by identifying each $p_i^-$ with the $p_i$ in $T^-$ and each $p_i^+$ with the $p_i$ in $T^+$.
In particular, for each gluing point $p$, the edge ending at $p$ and the edge starting at $p$ are merged to the same edge. In other words, $p$ becomes the middle point of an edge.
We obtain a connected directed graph $G$.
	\begin{figure}[H]
	    \centering
	    \includegraphics{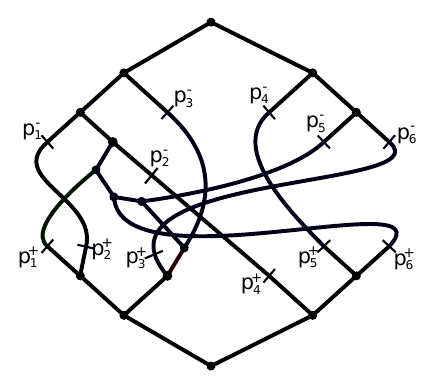}
	    \caption{The graph $G$.}
	\end{figure}
Observe that every maximal directed path in $G$ starts at the root of $T^-$ and ends in the root of $T^+$, and there is precisely one edge on it that lies between a split and a merge.
Cut every edge of $G$ connecting a split to a merge, let $q_1,\dots,q_m$ be these cutting points.
	\begin{figure}[H]
	    \centering
	    \includegraphics{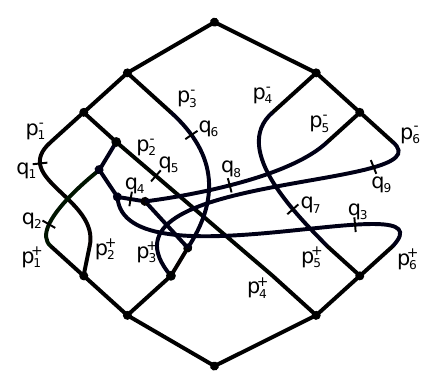}
	    \caption{The cutting points in $G$.}
	\end{figure}
Now we have two connected components, $T_1$ and $T_2$, with the property that all inner vertices of $T_1$ except the root of $T^- \subset T$ are splits and all inner vertices of $T_2$ except the root of $T^+ \subset T_2$ are merges. 
Every cut point $q$, is split to a leaf $q^-$ of $T_1$ and a leaf $q^+$ of $T_2$.
\begin{figure}[H]
    \centering
    \subfloat[The tree $T_1$]{\includegraphics[scale=1.1]{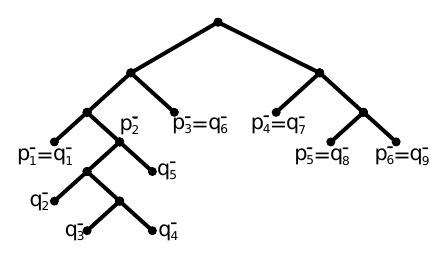}}%
    \qquad
    \subfloat[The tree $T_2$]{\includegraphics[scale=1.1]{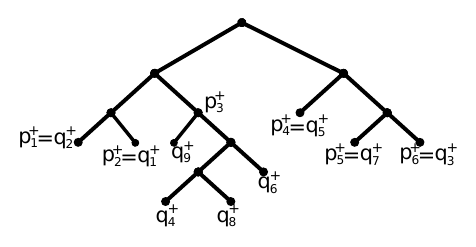}}%
    \caption{The trees in the tree pair.}%
    \label{fig:the trees in the tree pair}%
\end{figure}
Define $\kappa:\LL T_1 \to \LL T_2$ by $\kappa(q^-) = q^+$. 
The plane order on $T_1$ and $T_2$ is inherited from $T$ and $r$.
	
\emph{Claim 1:} $T_i$ are trees.

 Observe that $T_1$ consists of all paths in $G$ that start in the root of $T^+$ and end in a cut point $q$.
 Therefore $T_1$ is connected. To show that it does not have any loops, note that every undirected loop has splits and merges, which is not possible in $T_1$. Therefore a loop in $T_1$ would have to be a split loop. But this is impossible since no edge in a split loop can lie on a path from the root to a leaf. A similar argument works for $T_2$.
 
\emph{Claim 2:} $T = T_1 \cap T_2$.

The inclusion $\subseteq$ is obvious.
If this inclusion is strict, there has to be leaf $p\in \LL T$ that is an inner point of $T_1 \cap T_2$.
But then the edge in $D'$ starting in $p^-$ ends in a split, while the edge ending in $p^+$ starts by a merge. This cannot happen by Proposition~\ref{prop: loops in reduced diagram}.

\emph{Claim 3:} $(D,r,c)$ is the basic BM-diagram of $[\kappa,T_1,T_2]$.

This follows directly from the construction. Note that nowhere in the process did we modify $r$.

\emph{Claim 4:} The tree pair $[\kappa,T_1,T_2]$ is a revealing pair.

Note that a component of $T_2 \setminus T_1$ is isomorphic as plane ordered tree to a connected component of $D'$ that has only merges, and following the orbit of the root is the same as travelling along the corresponding directed cycle in $D$.
\end{proof}

\begin{remark}
Examining the construction from the proof of Proposition \ref{prop: a reduced BM-diagram comes from a revealing pair} we see that $P=[\kappa,T_1,T_2]$ satisfies the following. Every merge loop of $D$ with $\alpha$ merges and cutting class value $\mu$ corresponds to an attracting point in $P$ with attracting length $\alpha$ and period $\mu$.
Similarly every split loop of $D$ with $\rho$ splits and cutting class value $\nu$ corresponds to a repelling point in $P$ with repelling length $\rho$ and period $\nu$.
Every free loop of $D$ with cutting class $\omega$ corresponds to a periodic maximal chain in $P$ of length $\omega$.
\end{remark}

\section{Elliptic-hyperbolic decomposition} \label{sec: ell-hyp decomposition}

		In this section $\T = \T_{d,k}$ for $d \geq 2$ and $k \geq 1$.
	
	Le Boudec and Wesolek divide tree almost automorphisms into elliptic elements and translations, mimicking the division in $\AutT$, see Section 3 in \cite{lbw19}. However, while translations in $\AutT$ act on $\dT$ as one might expect from the term - there is one attracting point, one repelling point, and all other boundary points travel from the repelling to the attracting point - things in $\AAutT$ are more complicated. A translation can have several attracting and repelling points in the boundary, each with a different translation length. Those points may not even be fixed, but could have finite orbits. Points around one repelling point can distribute themselves to several attracting points. On top of that, looking at some balls might even give the impression that we are not dealing with a translation at all, as they will return to themselves again and again. In this section we try to shed light on the possible dynamic behaviour of tree almost automorphisms.
	We define a notion of hyperbolic elements in $\AAutT$, which will be a subset of Le Boudec's and Wesolek's translations. They will be those translations that show only trivial elliptic behaviour. We show that every element $g$ admits a unique decomposition $g=g_e g_h$ into an elliptic element $g_e$ and a hyperbolic element $g_h$ having disjoint supports. Towards the end of the section we also prove that for two elements to be conjugate, it is essentially enough if both of their factors are conjugate.
	
	\subsection{Dynamic characterization of boundary points}
	
	For a tree almost automorphism $g$ we examine the different kinds of boundary points with respect to the dynamics of $g$.
	
	\begin{definition} \label{def:characertisationofbdrypoints}
		Let $g \in \AAutT$ and $\eta \in \dT$. We call $\eta$
		\begin{enumerate}
			\item an \emph{attracting point} for $g$ if for every \nh~$B$ of $\eta$ there exists a neighbourhood $U\subseteq B$ of $\eta$ and an integer $n>0$ such that $g^n(U)\subsetneq U$. \label{g attracting}
			\item a \emph{repelling point} for $g$ if for every \nh~$B$ of $\eta$ there exists a neighbourhood $U\subseteq B$ of $\eta$ and an integer $n>0$ such that $g^n(U)\supsetneq U$. \label{g repelling}
			\item a \emph{stable point} for $g$ if for every \nh~$B$ of $\eta$ there exists a neighbourhood $U\subseteq B$ of $\eta$ and an integer $n>0$ such that $g^n(U)=U$. \label{g stable}
			\item a \emph{wandering point} for $g$ if there exists a \nh~$U$ of $\eta$ such that $g^n(U)\cap U=\emptyset$ for every $n>0$. \label{g wandering}
		\end{enumerate}
	    We denote the sets of attracting, repelling, stable and wandering points for $g$ by $\Att(g)$, $\Rep(g)$, $\operatorname{St}(g)$ and $\Wan(g)$.
	\end{definition}

    \begin{remark}\label{rem:pointcharakterisationofpowers}
    	It is obvious from the definition that $\Att(g)=\Att(g^{k})$, $\Rep(g)=\Rep(g^k)$ and $\Att(g)=\Rep(g^{-k})$ for all $k > 0$. Also we can easily see that $\operatorname{St}(g)=\operatorname{St}(g^{k})$ and therefore also $\Wan(g)=\Wan(g^{k})$ for all integers $k \neq 0$.
    \end{remark}
	
	We show that the possibilities from Definition \ref{def:characertisationofbdrypoints} are mutually exclusive and cover the whole boundary.
	
	\begin{proposition}\label{prop:decomposition of dT}
		Let $g \in \AAutT$ and $\eta\in\dT$. Then, $\eta$ is either attracting, repelling, wandering or stable for $g$, and these possibilities are mutually exclusive. Furthermore, $\Att(g)$ and $\Rep(g)$ are finite, $\Wan(g)$ is open, and $\operatorname{St}(g)$ is clopen. Consequently $\overline{\Wan(g)}=\Wan(g) \cup \Att(g) \cup \Rep(g)$.
	\end{proposition}

	This proposition follows directly from the following lemma connecting the different points of the boundary to revealing pairs. The basic idea of this lemma is already present in \cite{sd10}, Proposition 2.
	Recall the relevant terms given in Definition \ref{def:maximal_chains}.

	\begin{lemma}\label{lem:leaves of a revealing pair}
		Let $g \in \AAutT$ and let $\tp{g}{T_1}{T_2}$ be a revealing pair associated with $g$.
		Let $\foris{\phi}{T_1}{T_2}$ be the corresponding representative for $g$.
		Let $v$ be a leaf of $T_1$.
		\begin{enumerate}
			\item If $v$ is a periodic leaf, then $\dT_v \subset \operatorname{St}(g)$.
			\item If $v$ is in a wandering chain, then $\dT_v \subset \Wan(g)$.
			\item \label{item:attractorlength} If $v$ is in an attractor chain, then $\dT_v$ contains a unique attracting point $\eta$, and $\dT_v \setminus \{\eta\} \subset \Wan(g)$.
			
			More precisely, let $m$ be the period of the attractor. Then $\eta$ is the boundary point defined by the sequence  $(\phi^{km}(v))_{k \in \mathbb{N}}$.
			\item If $v$ is in a repeller chain, then $\dT_v$ contains a unique repelling point $\eta$, and $\dT_v \setminus \{\eta\} \subset \Wan(g)$.
			
			More precisely, let $m$ be the period of the repeller. Then $\eta$ is the boundary point defined by the sequence $(\phi^{-km}(v))_{k \in \mathbb{N}}$.
		\end{enumerate}
	\end{lemma}

	\begin{proof}
		The first and second statements are obvious from the definitions.
		The last statement is equivalent to the third after replacing $g$ with $g^{-1}$ because of Remark \ref{rem:pointcharakterisationofpowers}.
		
		To prove the third statement, observe that $\dT_v \supsetneq \phi^m(\dT_v)=\dT_{\phi^m(v)} \supsetneq \phi^{2m}(\dT_v) = \dT_{\phi^{2m}(v)} \supsetneq \dots$ and $\bigcap_{k \geq 1} \phi^{mk}(\dT_v) = \{\eta\}$. 
		Moreover, every neighbourhood of $\eta$ contains $\phi^{mk}(\dT_v)$ for large enough $k$. Therefore $\eta$ is indeed an attracting point.
		On the other hand, let $\xi \neq \eta$ be a point in $\dT_v$.
		Then, there exists $n > 0$ such that $\xi \notin \phi^{nm}(\dT_v)$.
	Let $U=\dT_v \setminus \phi^{nm}(\dT_v)$.
	Note that $\phi^k(U)$ is disjoint from $\dT_v$ if $m$ does not divide $k$, and is contained in $\phi^{nm}(\dT_v)$ if $m$ does divide $k$. This shows that $\xi$ is indeed a wandering point.
	\end{proof}
	
	The previous lemma implies that the following are well-defined.
	
	\begin{definition}
	    Let $g \in \AAutT$. Let $\eta$ be an attracting point of $g$.
	    With the notations as in Item \ref{item:attractorlength} of the previous lemma we call $m$ the \emph{period} of $\eta$ and $\operatorname{dist}(v,\phi^m(v))$ the \emph{attracting length} of $\eta$.
	    The \emph{period} and \emph{repelling length} of a repeller are defined in a similar fashion.
	\end{definition}

	\subsection{Dynamic characterization of almost automorphisms}
	
	Now we classify tree almost automorphisms according to their dynamic behaviour.
	
	\begin{definition}\label{def: elliptic}
		We call $g\in\AAutT$ \textit{elliptic} if $\eell(g)=\dT$.
		We call $g\in\AAutT$ \textit{hyperbolic} if it is not the identity and  $g|_{\operatorname{St}(g)}=id$. Denote by $\Ell$ and $\Hyp$ the sets of all elliptic and hyperbolic elements in $\AAutT$.
	\end{definition}
	
    Our definition of an elliptic element coincides with Definition 1.1 in \cite{lbw19}, see Lemma \ref{lem:ell=phiTT}.
	
	\begin{remark}
	Note that $\Ell$ is a clopen subset of $\AAutT$ and $\Hyp$ is closed, but not necessarily open.
		Clearly the classes $\Ell$ and $\Hyp$ are invariant under conjugation in $\AAutT$.
	\end{remark}
	
	\begin{definition}\label{def:EH decomposition}
		Let $g\in\AAutT$.
		We define $g_e \in \Ell$ by $g_e|_{\eell(g)}=g$ and $g_e|_{\overline{\Wan(g)}}=id$.
		Similarly we define $g_h \in \Hyp$ by $g_h|_{\eell(g)}=id$ and $g_h|_{\overline{\Wan(g)}}=g$.
		We call the decomposition $g=g_eg_h$ the \textit{elliptic-hyperbolic (EH) decomposition of $g$}.
	\end{definition}
	
	It is easy to see that the EH decomposition is the unique way of writing an element as product of an elliptic element and a hyperbolic element with disjoint supports.
	It is not surprising that the decomposition is a homeomorphism onto its image.
	
	\begin{lemma}
	The map $\AAutT \to \Ell \times \Hyp, \, g \mapsto (g_e,g_h)$ is injective, continuous and closed.
	\end{lemma}
	
	\begin{proof}
	We denote the decomposition map by $f\colon \AAutT \to \Ell \times \Hyp, \, f(g):=(g_e,g_h)$, and the multiplication map by $m(g,h):= g h$. Note that $m \circ f = id$.
	
	Injectivity of $f$ is obvious as $f$ is a right-inverse of the multiplication map.
	
	Continuity can be checked separately on $f_e(g):= g_e$ and $f_h(g):= g_h$.
	Let $h \in \Ell$ be an almost automorphism and let $\tp{h}{T}{T}$ be a revealing tree pair associated with $h$. (We use here the fact that every elliptic element admits an associated tree pair of this form, see \ref{lem:ell=phiTT} or \cite{lbw19}, Proposition 3.5.)
	Let $O$ be the set of all elliptic elements allowing a tree pair $\tp{h}{T}{T}$. Observe that $f_e^{-1}(O)$ consists of all almost automorphisms allowing a tree pair $P=[\kappa,T_1,T_2]$ such that the periodic leaves of $P$ are contained in $\LL T$ and $\kappa$ coincides with $\overline{h}$ on these periodic points. Together with Remark \ref{rem:tree pairs define topology} this shows that $f_e^{-1}(O)$ is open.
	The argument why $f_h$ is continuous is similar.
	
	To show that $f$ is closed, note that for every closed set $F \subset \AAutT$ holds $f(F)= m^{-1}(F) \cap \operatorname{Im}(f)$. It is therefore enough to show that the image of $f$ is closed.
	We will show that its complement is open. Observe that $\operatorname{Im}(f) = \{(g,h) \in \Ell \times \Hyp \mid \supp(g) \cap \supp(h) = \emptyset \}$ and let $(g,h) \notin \operatorname{Im}(f)$. If $g \notin \Ell$ or $h \notin \Hyp$, then $\Ell^c \times \AAutT$, $\AAutT \times \Hyp^c$ or $\Ell^c \times \Hyp^c$ is an open neighbourhood of $(g,h)$ disjoint from $\operatorname{Im}(f)$. Assume therefore that $g \in \Ell$ and $h \in \Hyp$ such that $\supp(g) \cap \supp(h) \neq \emptyset$.
	Since $\supp(h)$ is clopen, this implies that there exists an $\xi \in \dT$ such that $g(\xi) \neq \xi$ and $h(\xi) \neq \xi$, since otherwise $\{\xi \in \dT \mid g(\xi) \neq \xi \} \subset \supp(h)^c$ and this would imply $\supp(g) \subset \supp(h)^c$.
	Then there exists a vertex $x$ of $\T$ such that $g$ and $h$ have representatives $\phi,\psi$ such that $\phi(x) \neq x$ and $\psi(x) \neq x$. Let $O_1$ be the set of all almost automorphisms having a representative mapping $x$ to $\phi(x)$ and $O_2$ the set of all almost automorphisms having a representative mapping $x$ to $\psi(x)$. Then $O_1 \times O_2$ is an open neighbourhood of $(g,h)$ disjoint from $\operatorname{Im}(f)$.
	\end{proof}

	Note that $\supp(g_e)=\eell(g)\cap \supp(g)$ and $\supp(g_h)=\overline{\Wan(g)}$.
	The next lemma shows that the conjugacy problem on $\AAutT$ can be reduced to each of the classes $\Ell$, $\Hyp$ separately.
	
	\begin{proposition}\label{prop:conj_can_be_checked_on_ell_and_hyp_seperately}
		Let $g,f\in\AAutT$ and let $g=g_eg_h$ and $f=f_ef_h$ be their \ehdecom s.
		Then $g,f$ are conjugate in $\AAutT$ if and only if
		\begin{enumerate}
			\item $g_e$ is conjugate to $f_e$ and $g_h$ is conjugate to $f_h$; and
			\item either $\supp(f)=\supp(g)=\dT$, or both $\supp(f)\neq \dT$, $\supp(g)\neq \dT$.
		\end{enumerate}
	\end{proposition}
	
	\begin{proof}
		The ``only if'' direction is obvious because $(a^{-1}ga)_e = a^{-1}g_e a$ and $(a^{-1}ga)_h = a^{-1}g_h a$.
		
		For the ``if" direction, let $a,b\in\AAutT$ be such that $f_e =a^{-1}g_e a$ and $f_h =b^{-1}g_h b$.
		Denote $A := \supp(f_e)$, $B := \supp(f_h)$, $C := \supp(f)^c$ and $A' := \supp(g_e)$, $B'=\supp(g_h)$, $C'=\supp(g)^c$. Note that by Proposition \ref{prop:decomposition of dT} we have that $A\sqcup B\sqcup C=A'\sqcup B'\sqcup C'=\dT$ are both disjoint unions, and the sets $A\sqcup C, A'\sqcup C', B$ and $B'$ are clopen sets. Furthermore, we have $a(A)=A'$ and $b(B)=B'$.
		
		We first construct an element $a'$ with $f_e = a'^{-1}g_e a'$ and $a'(B)=B'$. 
		Both $B$ and $B'$ can be assumed to be non-empty, as otherwise it would imply that $g=g_e$ and $f=f_e$ and so there is nothing to prove.
		We can also assume that both $C$ and $C'$ are non-empty, as otherwise $a(B)=a(A)^c=A'^c=B'$ and so we can take $a'=a$.
		Under these assumptions, $B,B'\subset \T$ are clopen non-empty proper subsets. By Lemma \ref{lem: translating balls}(\ref{item:number of balls is preserved mod d-1}) the sets $B$ and $a^{-1}(B')$ consist of the same number of disjoint balls $\mod{d-1}$, since $a^{-1}b(B)=a^{-1}(B')$.
		As $A\subset (B\cup a^{-1}(B'))^c$ is proper, we can find a clopen set $W$ such that $A\subset W\subseteq(B\cup a^{-1}(B'))^c$.
		By Lemma \ref{lem: translating balls}(\ref{item:can map U1 to U2 while fixing W}), there exists an element $c\in \AAutT$ such that $c|_W = id$ and $c(a^{-1}B') = B$.
		Note that $\supp(c)$ is disjoint from $\supp(f_e)\subset W$ and so $c$ and $f_e$ commute. Defining $a'=ac^{-1}$, we have that $f_e=a'^{-1}g_e a'$, and moreover $a'(B)=B'$.
		
		We now have that $a'(A)=A',a'(B)=B'$ and it follows that $a'(C)=C'$. Since by assumption also $b(B)=B'$, the following element of $\AAutT$ is well defined:
	$c'=
	\begin{cases}
	{
	\renewcommand{\arraystretch}{0.9}
	\begin{array}{cc}
	a' & A\cup C \\
		b & B 
	\end{array}.
	}
	\end{cases}
		$
		We claim that $c'^{-1}gc'=f$. 
		
		Indeed, let $\eta\in\dT$.
		If $\eta\in A$, then $c'(\eta)=a'(\eta)\in A'$ and so $c'^{-1}gc'(\eta)=a'^{-1}g_e a'(\eta)=f_e(\eta)=f(\eta)$. Next, suppose $\eta\in B$. In this case $c'(\eta)=b(\eta)\in B'$ and so $c'^{-1}gc'(\eta)=b^{-1}g_hb(\eta)=f_h(\eta)=f(\eta)$. Lastly, suppose $\eta\in C$. Then since $c'(\eta)=a'(\eta)\in C'$ we have $c'^{-1}gc'(\eta)=\eta=f(\eta)$.

		\end{proof}
		
		\begin{example}
		The example in Fig. \ref{fig:support condition necessary} shows that the second condition in Proposition \ref{prop:conj_can_be_checked_on_ell_and_hyp_seperately} is indeed necessary.
		 \begin{figure}
  \centering
    \subfloat[An element $f$ with full support]{\includegraphics[scale=1]{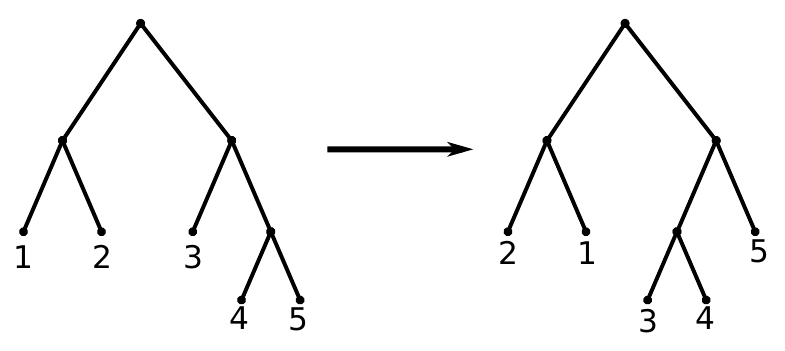}
    }
    
    \subfloat[An element $f$ with non-full support]{\includegraphics[scale=0.9]{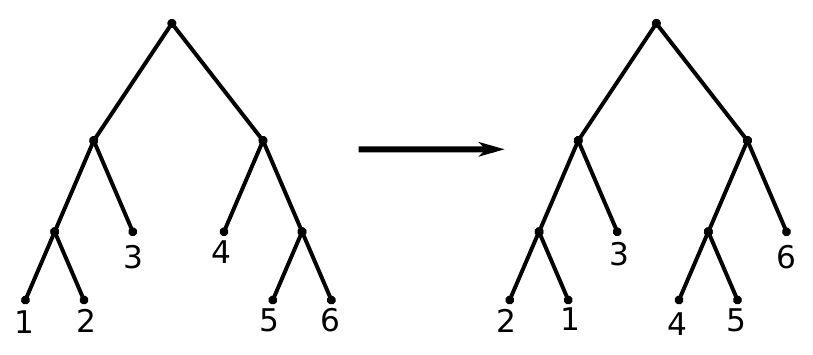}}
  \caption{We have $f_e=g_e$ and $f_h$ is conjugate to $g_h$, but $f$ and $g$ are not conjugate since $f$ has full support, but $g$ does not.}
  \label{fig:support condition necessary}
 \end{figure}
		\end{example}
	
       \begin{remark}
       There is nothing special about $\AAutT$, elliptic or hyperbolic elements in the proof for the preceding lemma. The only thing we use is that $\AAutT$ is a topological full group, that admits a unique decomposition of each element into two factors with disjoint clopen supports from disjoint conjugacy invariant sets.
       \end{remark}

\section{Elliptic elements} \label{sec: ell elements}

		In this section, $\T=\Tdk$ is again the tree such that the root has valency $k$ and all other vertices have valency $d+1$.
	
	Let $g$ be an elliptic element in $\AutT$. The dynamics of $g$ acting on $\T$ is described by a labeled graph, called the \textit{orbital type of $g$}. The orbital type is invariant under conjugation. In fact, conjugacy classes of elliptic elements in $\AutT$~ are classified by the orbital type: two elliptic elements are conjugate in $\AutT$~ if and only if they admit the same orbital type.
	
	In this section, we define the \textit{boundary orbital type} of an elliptic element in $\AAutT$.
	This will be an equivalence class of the orbital type of a forest isomorphism defining the elliptic element.
    Further, we show that two elliptic elements in $\AAutT$ are conjugate if and only if they admit the same boundary orbital type.
	
	Le Boudec and Wesolek give the following four characterisations of elliptic elements.
	
	\begin{lemma}[\cite{lbw19}, Proposition 3.5]\label{lem:ell=phiTT}
		Let $g \in \AAutTdk$. The following are equivalent.
		\begin{enumerate}
			\item There is a finite complete subtree $T$ of $\Tdk$ such that the tree pair $\tp{g}{T}{T}$ is associated to $g$.
			\item Some power of $g$ is a tree automorphism of $\Tdk$ fixing the root.
			\item The subgroup $\overline{\langle g \rangle} \leq \AAutTdk$ is compact.
			\item The element $g$ is not a translation, i.e. there do not exist a ball $B\subset \partial \Tdk$ and an integer $n \geq 1$ such that $g^n(B) \subsetneq B$.
		\end{enumerate}
	\end{lemma}
	
	\subsection{Orbital type}
	
	In this subsection we extend the classical orbital type of elliptic tree automorphisms to elliptic forest automorphisms.
	
	\begin{definition}
		Let $\F$ be a forest. A \emph{labeling of $\F$} is a map $l \colon \Vertices (\F)\to \mathbb{N}_{>0}$ defined on the vertices of $\F$. The pair $(\F,l)$ is called a \emph{labeled forest}.
		
		A forest isomorphism $f \colon \F_1\to \F_2$ between two labeled forests $(\F_1,l_1)$ and $(\F_2,l_2)$ is called an \emph{isomorphism of labeled forests} if $l_2(f(v))=l_1(v)$ for every $v\in \Vertices(\F_1)$.
	\end{definition}
	
	We often just write $(\F_1,l_1)=(\F_2,l_2)$ when we mean isomorphic as labeled forests.
		
	\begin{definition}[Orbital type]\label{def:OT}
		Let $T\subset\Tdk$ be a finite complete subtree and let $\varphi$ be an automorphism of the forest $\F:=\Tdk\setminus T$. Then, the \textit{orbital type of $\varphi$} is the labeled forest $\OT(\varphi) := (\overline{\F},l)$, where $\overline{\F}:=\overline{\langle\varphi\rangle} \setminus \F$ is the quotient graph,
		and the labeling map $l\colon \Vertices(\overline{\F})\to \mathbb{N}$ is defined by sending each equivalence class $[v] \in \Vertices(\overline{\F})$ with $v \in \Vertices(\F)$ to its cardinality $l([v]):=|\{\phi^n(v) \mid n \in \mathbb{Z}\}|$.
	\end{definition}
	
	An example is drawn in Fig. \ref{fig:OT}.
	
 \begin{figure}
  \centering
    \subfloat[a forest automorphism; the left two trees form a forest, and the arrow maps it to the forest consisting of the right two trees]{\includegraphics[width=\textwidth]{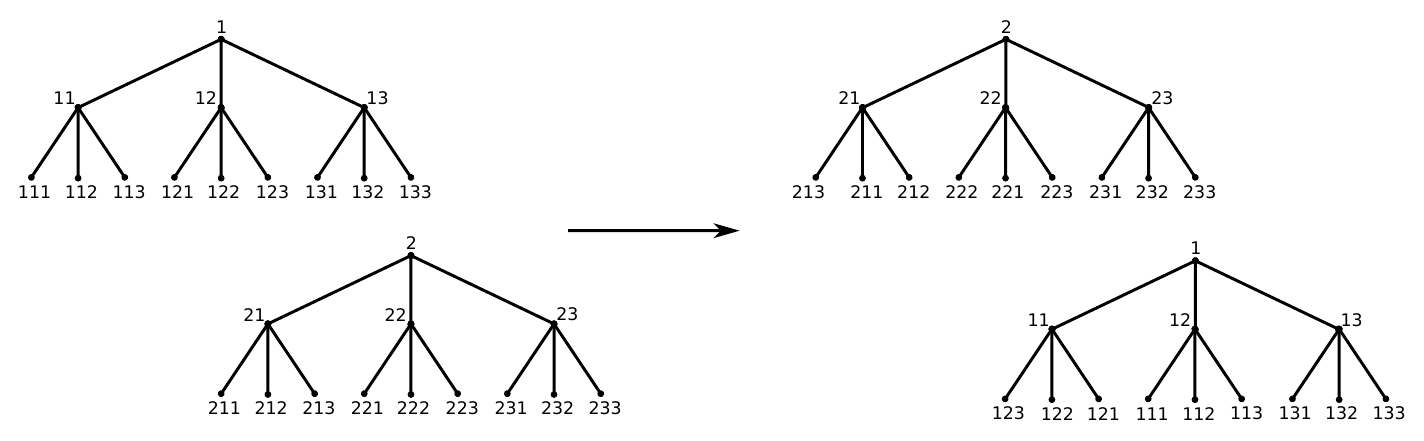}
    }
    
    \subfloat[its orbital type]{\includegraphics[scale=1.2]{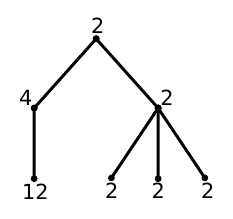}}
  \caption{A forest isomorphism and its orbital type.}
  \label{fig:OT}
 \end{figure}
	
	In case $\F=\T$ is a level homogeneous tree and $\varphi \in \AutT$, Definition \ref{def:OT} coincides with the definition of orbital type  given by Gawron, Nekrashevych and Sushchansky~\cite{gns01}. They give the following complete characterisation when two elliptic tree automorphisms are conjugate.
	
	\begin{theorem}[\cite{gns01}, Theorem 3.1 and Theorem 5.1]\label{OT preserved by conj in AutT}
		Let $\F=\T$ be a level homogeneous tree, and let $\varphi,\varphi'\in\AutT$ be two elliptic elements. Then $\varphi$ and $\varphi'$ are conjugate in $\AutT$ if and only if $\OT(\varphi)$ and $\OT(\varphi')$ are isomorphic as labeled trees.
	\end{theorem}
	
	To make use of this theorem when talking about almost automorphisms, we now describe a way how to get from elliptic almost automorphisms to elliptic automorphisms of a perhaps different tree.
	Let $T \subset \T$ be a finite complete subtree and $m := |\LL T|$.
	Let $p_T \colon \T \to \T_{d,m}$ be the map that contracts all the inner vertices of $T$ and the edges connecting them to a point.
	Then the restriction $p_T\restriction_{\T \setminus T} \colon \T \setminus T \to \T_{d,m} \setminus p_T(T)$ is a forest isomorphism. For an almost automorphism $\foris{\phi}{T}{T}$ define $i_T(\phi) := p_T \circ \phi \circ p_T^{-1}$.
	Note that $p_T(T)$ is the $1$-ball around the root of $\T_{d,m}$, so in fact $i_T(\phi) \in \Aut(\T_{d,m})$. Clearly the map $i_T \colon \Aut(\T \setminus T) \to \Aut(\Tdm)$ is an isomorphism.
	The following lemma says that $\OT(\phi)=\OT(i_T(\phi))$, where in this equation $i_T(\phi)$ is again viewed as a forest automorphism of $\Tdm \setminus p_T(T)$. We omit its proof as it is an easy exercise.
	
	\begin{lemma}\label{lem:OT preserved by rfi}
		Let $\F$ be a forest and let $\varphi\colon \F\to \F$ be an automorphism of $\F$.
		Suppose that $i\colon \F\to \F'$ is a forest isomorphism. Then $\OT(\phi)=\OT(i \circ \phi \circ i^{-1})$.
	\end{lemma}

We now determine which labeled forests may be obtained as orbital types of elliptic almost automorphisms.
Recall that a \emph{rooted forest} is a forest where each connected component is a rooted tree. Note that if $T \subset \T$ is a finite complete subtree, then $\T \setminus T$ has a natural structure as a rooted forest, namely by taking $\LL T$ as the set of roots.

\begin{remark}\label{rem:properties_of_labels_of_OT}
Let $T$ be a complete finite subtree of $\T=\Tdk$, and let $\F := \T \setminus T$.
For any $\phi \in \Aut(\F)$ the labeled forest $\OT(\phi)$ satisfies the following:
\begin{enumerate}
    \item $\sum_{v \text{ root of }\F} l(v) = k + n'(d-1)$ for some $n' \geq 0$; \label{item:sum of labels of roots}
    \item $l(v)$ divides $l(u)$ for all vertices $v,u$ such that $u$ is a descendant of $v$; and \label{item:l(v) divides l(u)}
    \item $\sum_{\text{$u$ child of $v$}} l(u)=d \cdot l(v)$ for every vertex $v$ of $\OT(\phi)$.
    \label{item:sum of labels of children}
\end{enumerate}
\end{remark}

In the other direction, we have the following.

\begin{lemma}\label{lemma:dk-type iff OT of AAutTdk}
Let $(\F,l)$ be a rooted labeled forest satisfying Items \ref{item:sum of labels of roots}, \ref{item:l(v) divides l(u)} and \ref{item:sum of labels of children} of the previous remark.
Then for every complete finite subtree $T$ of $\T=\T_{d,k}$ with $|\LL T|=\sum_{v \text{ root of }\F} l(v)$ there exists a forest automorphism $\foris{\phi}{T}{T}$ with $\OT(\phi)=(\F,l)$.

\end{lemma}

\begin{proof}
By Item \ref{item:sum of labels of roots} there exists a complete finite subtree $T$ of $\Tdk$ with $|\LL T|=m=\sum_{v \text{ root of }\F} l(v)$.
By the proof of Theorem 3.1 in \cite{gns01} there exists a $g \in \Aut(\Tdm)$ 
such that $(\F,l)=\OT(g)\setminus B$, where $B$ is the ball of radius $1$ around the root.
Recall that the map $i_T \colon \Aut(\T \setminus T) \to \Aut(\Tdm \setminus p_T(T))$ from above induces an orbital type preserving isomorphism, so $\phi=i_T^{-1}(g)$ does the job.
\end{proof}

\subsection{Boundary orbital type and conjugacy}

Now we define an ``almost''-version of the orbital type of a forest automorphism and show that it completely determines the conjugacy class of the corresponding elliptic almost automorphism.
A subforest $F$ of a forest $\F$ is called \emph{complete} if it is a union of complete trees. Our forest $\F$ will always be rooted and unless explicitly stated otherwise we assume that these complete trees are empty or contain a root of $\F$.
	
\begin{definition}
	Let $(\F_1,l_1)$ and $(\F_2,l_2)$ be two labeled forests as in Remark~\ref{rem:properties_of_labels_of_OT}. We call them \emph{boundary equivalent} if there exist finite complete subforests $F_i\subset \F_i$, $i=1,2$ such that $\F_1\setminus F_1$ and $\F_2 \setminus F_2$, equipped with the restrictions of $l_1,l_2$, are isomorphic as labeled forests.
	Let $\varphi$ be an automorphism of a forest $\F$. The equivalence class of the labeled forest $\OT(\varphi)$ is called \emph{the boundary orbital type of $\varphi$}, and is denoted by $\BOT(\varphi)$.
\end{definition}

We ignore the subtlety that, strictly speaking, these "equivalence classes" are not sets, like the class of all trees is too big to be a set.
	
	Let $\T$ be a tree and let $g\in\AAutT$ be an elliptic element. If $\foris{\phi}{T}{T}$ and $\foris{\phi}{T'}{T'}$ are two forest automorphisms representing $g$, then both $\varphi,\varphi'$ are defined on $\T\setminus (T\cup T')$ and equal there, and so $\BOT(\varphi)=\BOT(\varphi')$. It follows that the following is well-defined.
	
	\begin{definition}[Boundary orbital type]\label{def:Boundary orbital type}
		 The \emph{boundary orbital type} of an elliptic tree almost automorphism $g$, denoted $\BOT(g)$, is defined to be the boundary orbital type of one (and therefore all) of its representatives.
	\end{definition}
	
	We show that the boundary orbital type fully characterizes conjugacy of elliptic elements.
    First we show the perhaps surprising fact that the orbital type of a forest automorphism contains information about the number of trees in the forest.
	\begin{lemma}\label{cor:OT gives number of leaves}
		Suppose $\foris{\phi}{T}{T}$ and $\foris{\phi'}{T'}{T'}$ are forest isomorphisms with	 the same orbital type. Then $|\LL T |=|\LL T'|$.
	\end{lemma}

\begin{proof}
	Let $\foris{\phi}{T}{T}$ be an automorphism of the forest $\F=\TwoT$ and consider the labeled graph of orbits $\OT(\phi)$. For a vertex $v\in \F$ denote by $[v]$ its image in $\OT(\phi)$. Every root of $\F$ (namely, every leaf of $T$) is mapped to a vertex in $\OT(\phi)$ whose label is minimal in its connected component. Moreover, if $r$ is a root of $\F$ then there are exactly $l([r])$ roots of $\F$ that are mapped to $[r]$. It follows that $|\LL T |=\sum_{C} \min\{l([v])|[v]\in C\}$, where the sum runs over all connected components of $\OT(\phi)$. 
\end{proof}
		
	\begin{theorem}\label{thm:conjugacy_elliptic}
		Let $\T=\Tdk$. Let $g,g'\in\AAutT$ be two elliptic elements, with boundary orbital types $\BOT(g)$ and $\BOT(g')$. Then $g,g'$ are conjugate in $\AAutT$ if and only if $\BOT(g)=\BOT(g')$.
	\end{theorem}
	
	\begin{proof}
		For the "only if"-direction, suppose $g'=aga^{-1}$ for some $a\in \AAutT$.
		Let $\foris{\psi}{T}{T}$ and $\foris{\phi}{T_1}{T_2}$ be forest isomorphisms representing $g$ and $a$. Without loss of generality we can assume that $T_1=T$. By Lemma \ref{lem:OT preserved by rfi}, $\OT(\psi)=\OT(\varphi\psi\varphi^{-1})$ and so in particular,  $\BOT(g)=\BOT(aga^{-1})$.
		
		Now we show the "if"-direction. Suppose $\BOT(g)=\BOT(g')$.

		\emph{Step 1:}  There exist finite complete trees $T$, $T'$ of $\T$ and forest automorphisms $\foris{\psi}{T}{T}$, $\foris{\psi'}{T'}{T'}$ representing $g$ and $g'$, such that $\OT(\psi)=\OT(\psi')$.
			
				Indeed, let $\foris{\psi_0}{T_0}{T_0}$, $\foris{\psi_0'}{T'_0}{T'_0}$ be any forest automorphisms representing $g$ and $g'$.
				Since $\BOT(g)=\BOT(g')$, there exist finite complete subforests $\bar{D}\subset\OT(\psi_0)$ and $\bar{D'}\subset\OT(\psi'_0)$ such that $\OT(\psi_0)\setminus\bar{D}$ and $\OT(\psi'_0)\setminus\bar{D'}$ are isomorphic as labeled graphs.
				Note that $\bar{D}$ (respectively, $\bar{D'}$) is a union of complete finite trees, and so its preimage $D\subset\TwoT$ (resp. $D'\subset\TwoT'$) is a union of complete finite trees, with roots in $\LL T $ (resp. $\LL T'$).
				In particular, $T\cup D$ (resp. $T'\cup D'$) is a complete finite subtree of $\T$.
				Let $\psi$ denote the restriction of $\psi_0$ to the forest $\T\setminus (T\cup D)$, and similarly $\psi'$ the restriction of $\psi_0'$ to $\T\setminus (T'\cup D')$. Then indeed $\psi$ and $\psi'$ represent $g$ and $g'$, and $\OT(\psi)=\OT(\psi')$.
		
		\emph{Step 2:} Up to replacing $g'$ by a conjugate, we can assume $T=T'$.
			
				By the previous step $\OT(\psi)=\OT(\psi')$. Lemma \ref{cor:OT gives number of leaves} implies that $T$ and $T'$ have the same number of leaves and therefore there exists a forest isomorphism, $\foris{\chi}{T'}{T}$. Then $\foris{\chi\psi' \chi^{-1}}{T}{T}$ represents a conjugate of $g'$.
			
			\emph{Step 3:} The forest isomorphisms $\psi$ and $\psi'$ are conjugate by an automorphism of the forest $\TwoT$.
			
			    Let $m := |\LL T|$.
				 It then follows from Theorem \ref{OT preserved by conj in AutT} that $\psi_1$ and $\psi_1'$ are conjugate in $\AutTdm$. Let $\phi_1 \in\AutTdm$ be such that $\phi_1\psi_1\phi_1^{-1}=\psi_1'$. Let $\phi_0$ be the restriction of $\phi_1$ to $\Tdm\setminus B_1(r)=\T_{d,m} \setminus p_T(T)$ and denote by $\phi=p_T^{-1}\phi_0 p_T$ the corresponding automorphism of $\TwoT$. Then $\phi \psi \phi^{-1}=\psi'$.
	This concludes the proof of the theorem.
	\end{proof}
	
\begin{example}
    Figure \ref{fig:elliptic tree aut conj in aaut} shows an example of two elliptic automorphisms of $\T_{2,2}$ that are conjugate in $\AAut(\T_{2,2})$, but not in $\Aut(\T_{2,2})$.
     \begin{figure}
  \centering
    \subfloat[This tree automorphism has three fixed vertices.]{\includegraphics[scale=0.8]{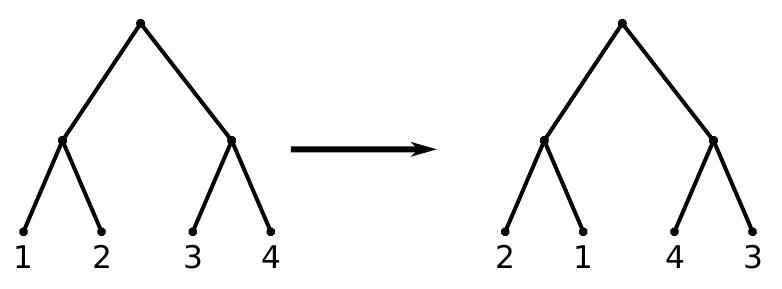}
    }
    
    \subfloat[This tree automorphism has only one fixed vertex.]{\includegraphics[scale=0.8]{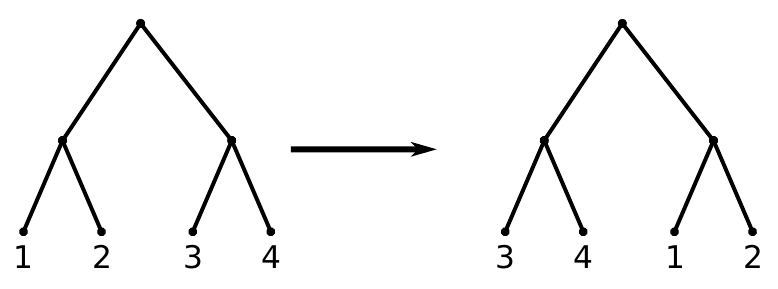}}
    
    \subfloat[This is the almost automorphism via which the two are conjugate.]{\includegraphics[scale=0.8]{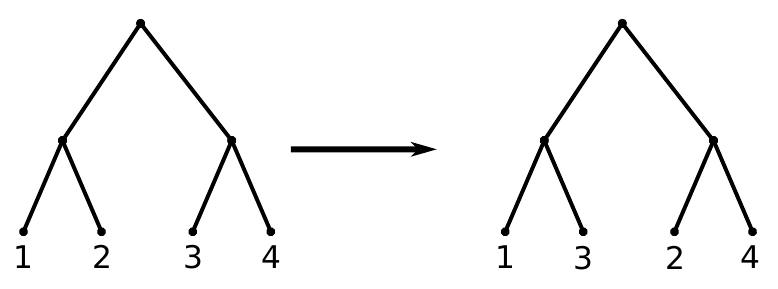}}
  \caption{They cannot be conjugate in $\Aut(\T_{2,2})$ because they have differently many fixed vertices.}
  \label{fig:elliptic tree aut conj in aaut}
 \end{figure}
\end{example}
	
\begin{remark}\label{rem:all_reps_of_BOT}
Let $g$ be an almost automorphism of $\T$ and $\F \in \BOT(g)$.
In this remark we want to explain when one can find an $\F' \in \BOT(g)$ with $\F \subset \F'$. It is enough to look at the labels of the roots.
Recall Remark \ref{rem:properties_of_labels_of_OT}. 
By Item \ref{item:sum of labels of roots}, the sum of all labels of roots in $\F$ is of the form $k + n'\cdot(d-1)$ with $n' \geq 0$. If it is equal to $k$ already, we are done, there is no possible bigger forest. 
Otherwise, by Item \ref{item:l(v) divides l(u)}, any subset of roots with labels $md_1, \dots, md_n$ satisfying $\sum d_i=d$, can be connected to a new root with label $m$, provided the sum of all labels of roots does not become smaller than $k$.
By Lemma \ref{lemma:dk-type iff OT of AAutTdk} there will be a forest isomorphism realizing this labeled forest.
 \begin{figure}[H]
  \centering
    \includegraphics[scale=1]{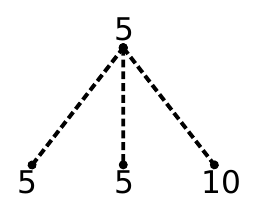}
  \caption{If $d=4$ and $k\geq 4$, since $1+1+2=4$, the three roots with label $5\cdot 1,5\cdot 1$ and $5 \cdot 2$ can be connected to a new vertex with label $5$.}
 \end{figure}
\end{remark}	

\subsection{Closure of conjugacy classes}

We give a characterization for the question when an element $g$ is contained in the closure of the conjugacy class of an element $h$. We denote the conjugacy class of $h$ by $[h]$.

For a rooted, labeled forest $\F$ let $s(\F)$ denote the multiset of labels of roots of $\F$.

\begin{proposition}  
Let $g,h$ be elliptic elements in $\AAutTdk$.
The following are equivalent.
\begin{enumerate}
    \item The element $g$ belongs to the closure of the conjugacy class of $h$.
    \item For every $\F\in \BOT(g)$ there exists $\Tilde{\F}\in \BOT(h)$ such that $s(\F)=s(\Tilde{\F})$.
    \item For almost every $\F\in \BOT(g)$ there exists $\Tilde{\F}\in \BOT(h)$ such that $s(\F)=s(\Tilde{\F})$.
\end{enumerate}
\end{proposition}

\begin{proof}
We first show that 1 implies 2.
Suppose first that $g \in \overline{[h]}$ and let $(h_n)_n$ be a sequence of conjugates of $h$ converging to $g$. Consider a labeled forest $\F\in \BOT(g)$. By Lemma \ref{lemma:dk-type iff OT of AAutTdk} and Theorem \ref{thm:conjugacy_elliptic} there exists a forest isomorphism $\foris{\varphi}{T}{T}$ representing a conjugate $g' = aga^{-1}$ such that $\F=\OT(\varphi)$. Note that $ah_na^{-1}\to g'$. It follows that there exists an integer $N \geq 0$ such that for all $n\geq N$ the element $ah_na^{-1}$ has a representative $\foris{\psi_n}{T}{T}$ such that $\psi_n \restriction_{\LL T }=\varphi\restriction_{\LL T }$. Take $\Tilde{\F}=OT(\psi_N) \in \BOT(h_N)=\BOT(h)$, then $s(\F)=s(\Tilde{\F})$.

It is obvious that 2 implies 3.

Now we prove that 3 implies 1.
Let $\BOT(g)'$ be the set of all $\F\in \BOT(g)$ such that there exists $\Tilde{\F}\in \BOT(h)$ with $s(\F)=s(\Tilde{\F})$.
By assumption $\BOT(g) \setminus \BOT(g)'$ is finite, so let $\foris{\phi_i}{T_{1,i}}{T_{2,i}}$ for $i=1,\dots,n$ be representatives of all elements in $\BOT(g) \setminus \BOT(g)'$. Let $T'$ be a finite complete subtree of $\T$ with $T_{1,1} \cup \dots \cup T_{1,n} \subset T'$.
By construction $\{\OT(\phi) \mid \phi \in \Aut(\T \setminus T), \, T \supset T' \text{ finite complete}, \, \phi \text{ represents }g\} \subset \BOT(g)'$.
We have to show that for every finite complete subtree $T \supset T'$ of $\T$ such that $g$ has a representative $\foris{\phi}{T}{T}$ there exists a forest isomorphism $\foris{\psi}{T}{T}$ representing a conjugate of $h$ with $\psi|_{\LL T} = \phi|_{\LL T}$.
Let $T\supset T'$ be such a tree and $\foris{\varphi}{T}{T}$ be a representative for $g$.
By assumption, there exists $\Tilde{\F}\in \BOT(h)$ with $s(\Tilde{\F})=s(\OT(\varphi))$. Lemma \ref{lemma:dk-type iff OT of AAutTdk} gives us that $\Tilde{\F}$ is the orbital type of some forest isomorphism $\foris{\psi'}{T}{T}$ representing a conjugate of $h$.
Recall that $\psi'|_{\LL T}$ is a permutation of finitely many elements. Hence it is a product of finitely many disjoint cycles and the leghths of these cycles are precisely the elements of $s(\Tilde{\F})$. Also recall that $s(\Tilde{\F})$ is a complete conjugacy invariant of the finite group $\Sym(\LL T)$.
Since for every permutation $\sigma \in \Sym(\LL T)$ there exists a forest isomorphism $\foris{\alpha}{T}{T}$ with $\alpha|_{\LL T}=\sigma$, it is possible to conjugate $\psi'$ to obtain an element $\foris{\psi}{T}{T}$ with $\psi\restriction_{\LL T }=\varphi\restriction_{\LL T }$. This finishes the proof.
\end{proof}

We conclude this section by considering the set of $\AutT$-conjugates.

\begin{proposition}\label{prop:elliptic AutT conjugates of T22}
If $d=k=2$, then $\{ gag^{-1} \mid a \in \Aut(\T_{2,2}), g \in\AAut(\T_{2,2})\}$ is closed. More precisely, an elliptic element $g \in \AAut(\T_{2,2})$ is conjugate to a tree automorphism if and only if for one (and hence every) forest $\F \in \BOT(g)$, the multiset of labels of roots $s(\F)$ only consists of powers of $2$.
\end{proposition}

\begin{proof}
We first show the ``only if'' direction. If an element of $s(\F)$ is divisible by an odd prime $p$, then almost every $\Tilde{\F} \in \BOT(g)$ has a root the label of which is divisible by $p$.
But for an automorphism of $\T_{2,2}$, all the orbit sizes
of all vertices are powers of $2$. Hence by Theorem \ref{thm:conjugacy_elliptic} we are done with this direction.

For the ``if''-direction, let $\F \in \BOT(g)$ be such that $s(\F)$ only consists of powers of $2$.
By Remark \ref{rem:all_reps_of_BOT} we can enlarge $\F$ either by connecting two trees with root labels $2^n$ and $2^n$ to a new root with label $2^n$, or by connecting one tree with label $2^n$ to a new root with label $2^{n-1}$. Both operations do not destroy the property that all labels of roots are powers of $2$, so we can continue until the sum of the labels is $2$ and we are done.
\end{proof}

\begin{corollary}
Let $d=k=2$ and let $g$ be an elliptic element conjugate to a tree automorphism. Then $id \in \overline{[g]}$ if and only if $\partial\F$ is infinite for one (and hence every) forest $\F\in \BOT(g)$.
\end{corollary}

\begin{proof}
The ``only if''-direction is obvious from Theorem \ref{thm:conjugacy_elliptic}.
For the ``if''-direction, note that for every large enough $m \geq 0$ there exists an $\F \in \BOT(g)$ with $m$ many connected components. For each component, we can enlarge it by a new root, the label of which is $1/2$ of the label of the previous root. We continue with this process until all roots are of label $1$, so we have precisely $m$ roots of label $1$. By Theorem \ref{thm:conjugacy_elliptic} we are done.
\end{proof}

\begin{remark}\label{rem:elliptic Aut(T) conjugates not closed}
For $d=k=3$ the set of all $\AutT$-conjugates is not closed, as we illustrate now by an example.
Figure \ref{fig:AutT33-conjugates not closed 2} shows a sequence $h_n$ of almost automorphisms, that converge to the element $g$ given in Fig. \ref{fig:AutT33-conjugates not closed 1}. While $h_n$ is conjugate to an element in $\AutT$ for all $n$, this is not the case for $g$.

More precisely, let $h_n$ be the Higman--Thompson element from Fig. \ref{fig:AutT33-conjugates not closed 2}(a).
  \begin{figure}
  \centering
    \subfloat[the almost automorphism $h_n$]{\includegraphics[scale=0.8]{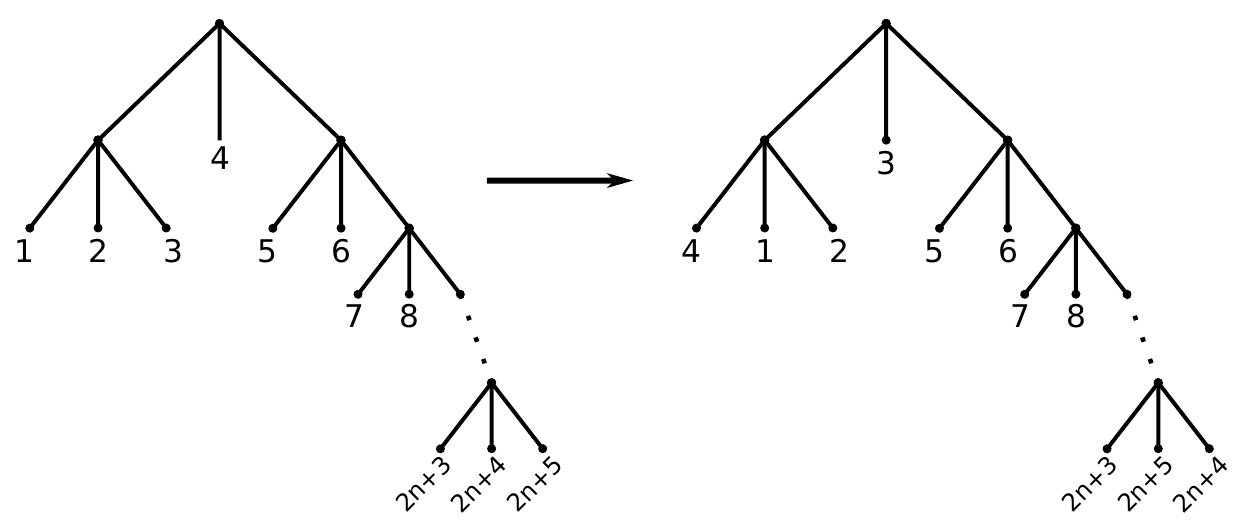}
    }
    
    \subfloat[its orbital type]{\includegraphics[scale=0.8]{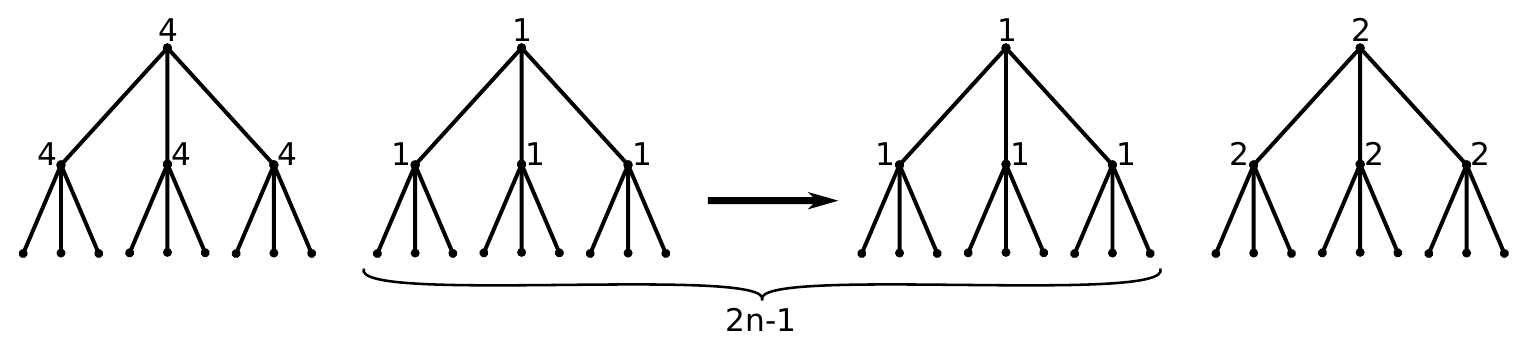}}
    
    \subfloat[a tree automorphism it is conjugate to]{\includegraphics[scale=0.8]{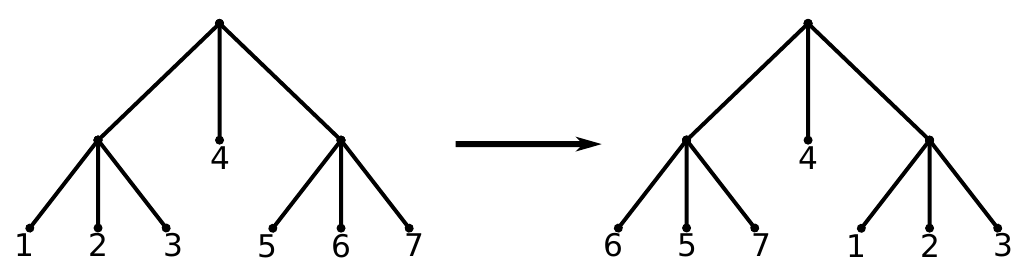}}
    
    \subfloat[its orbital type]{\includegraphics[scale=0.8]{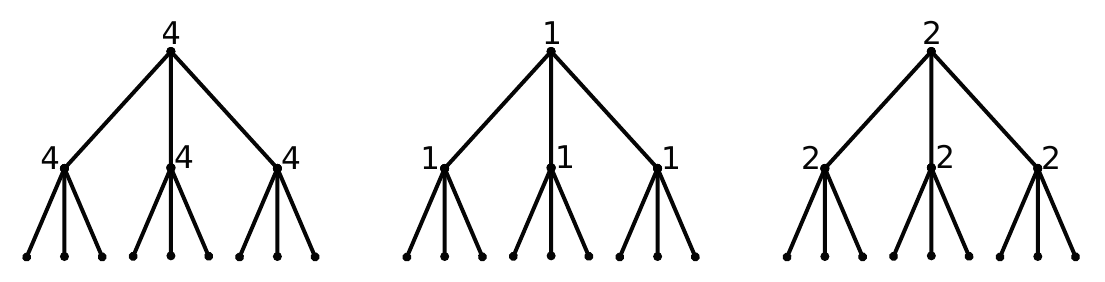}}
  \caption{The almost automorphisms in the first and on the third picture have the same boundary orbital type, so they are conjugate.}
  \label{fig:AutT33-conjugates not closed 2}
 \end{figure}
 Each $h_n$ is conjugate to the tree automorphism (also a Higman--Thompson element) depicted in Fig. \ref{fig:AutT33-conjugates not closed 2}(c);  they have the same boundary orbital type because in $\T_{3,3}$ the numbers of leaves of complete finite subtrees are exactly the odd numbers.
 
Let $g$ be the Higman--Thompson element from Fig. \ref{fig:AutT33-conjugates not closed 1}. It is clear that the sequence $(h_n)$ converges to $g$.
However, it is not hard to see that it can not be conjugate to a tree automorphism.
Indeed, in $\Aut(\T_{3,3})$, a leaf of orbit size $4$ must have an ancestor of orbit size $2$. Suppose now that an element $\F\in \BOT(g)$ contains a vertex of label $2$. It would either have exactly one child, labeled by $6$; have exactly two children, labeled $2$ and $4$; or it would have three children, all labeled by $2$. In all cases, $\F$ must either contain infinitely many vertices with labels divisible by $6$, or, it must contain infinitely many vertices labeled $2$. Both options contradict the assumption $\F$ is equivalent to $\OT(g)$. 
 \begin{figure}
  \centering
    \subfloat[the almost automorphism $g$]{\includegraphics[scale=0.8]{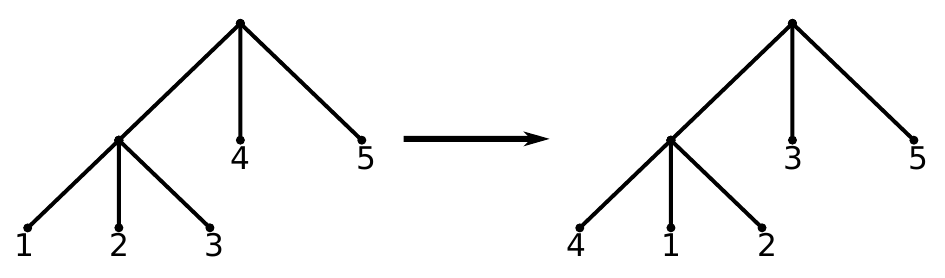}
    }
    
    \subfloat[its orbital type]{\includegraphics[scale=0.9]{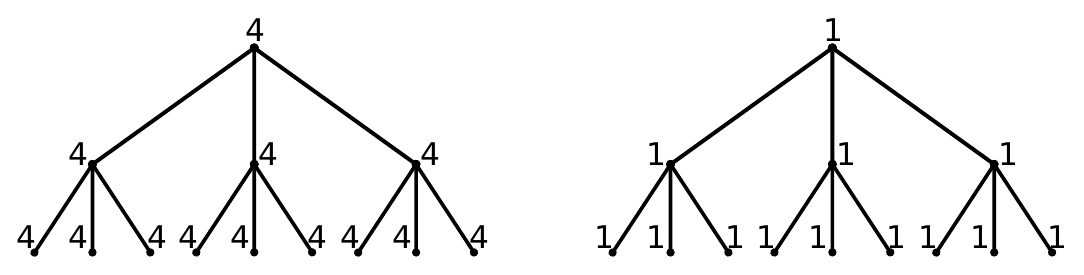}}
  \label{fig:AutT33-conjugates not closed 1}
 \end{figure}
\end{remark}

\begin{question}
For which $d$ and $k$ is the set of $\Aut(\Tdk)$-conjugates closed in $\AAutTdk$?
\end{question}
	
\section{Hyperbolic elements} \label{sec: hyp elements}
	
	In this section $\T = \T_{d,k}$ again denotes the tree such that the root has valency $k \geq 1$ and all other vertices have valency $d+1 \geq 3$. We fix a plane order on $\T$.
The main goal of this section is to prove that two hyperbolic elements are conjugate if and only if the *-reduced BM-diagrams
of sufficiently close Higman--Thompson elements
differ only in the rotation system.

\begin{theorem} \label{thm:hyperblolic_conjugacy}
    Let $\T = \T_{d,k}$.
	Let $g,h$ be hyperbolic tree almost automorphisms of $\T$. Then, $g$ and $h$ are conjugate if and only if their *-reduced BM-diagrams from a revealing pair differ only in the rotation system.
\end{theorem}

\subsection{Passing to Higman--Thompson elements}

The first step in the proof of Theorem \ref{thm:hyperblolic_conjugacy} is to show that a hyperbolic almost automorphism and a sufficiently close Higman--Thompson element are conjugate.

Let $x$ be a vertex in $\T$. Recall that $\T_x$ is a subtree of $\T$ that is rooted in $x$ and isomorphic to $\T_{d,d}$. For two vertices $x,y\in \T$ different from the root, the plane order of $\T$ induces a unique plane order preserving isomorphism $J_{x,y} \colon \T_x\to \T_y$. Whenever $\alpha$ is an automorphism of a tree fixing some vertex $x$, we denote by $\alpha_x \in \Aut(\T_x)$ the restriction of $\alpha$ to $\T_x$.

The following lemma is about recursively defining a tree automorphism.

\begin{lemma}
    \label{lem: constructing self periodic element}
    Let $x$ be a vertex of $\T$ and let $y$ be a descendant of $x$. Let $\eta \colon \T_x \to \T_y$, $\xi \colon \T_y \to \T_x$ and  $\alpha \in \Aut(\T_x \setminus \T_y)$ be isomorphisms. Then there exists an automorphism $\beta \in \Aut(\T_x)$ such that $\beta(y)=y$ and
    $$ \beta \restriction_{\T_x \setminus \T_y}=\alpha$$
    $$ \beta \restriction_{\T_y}=\eta \beta \xi.$$
\end{lemma}

\begin{proof}
The proof is by recursion. First set $\beta\restriction_{\T_x\backslash\T_y}=\alpha$, then set $\beta\restriction_{\T_y\backslash \eta(\T_y)}=(\eta\xi)\restriction_{\T_x\backslash\T_y}$, and so forth. It is a simple exercise to see this gives a well defined automorphism.
\end{proof}

\begin{proposition}\label{prop:v and av are conj}
    Let $g, h \in \AAutT$ be two hyperbolic elements that admit the same revealing pair $\tp{g}{T_1}{T_2}=\tp{h}{T_1}{T_2}$. Then there exists $b \in \Fix(T_1\cup T_2)$ such that $g = b^{-1}hb$.

    Equivalently, if $v$ is the Higman--Thompson element induced by a revealing pair $[\kappa,T_1,T_2]$ and if $a\in \Fix(T_2)$ acts trivially on $\eell(v)$, then there exists  $b \in \Fix(T_1\cup T_2)$ such that $av=b^{-1}vb$.
\end{proposition}

\begin{proof} We prove the second formulation of the proposition.

Observe that $av=b^{-1}vb$ is equivalent to $v^{-1}bav=b$. Let $\foris{\phi}{T_1}{T_2}$ be the (unique) plane order preserving forest isomorphism such that $\varphi\restriction _{\LL T_1}=\kappa$. By assumption $\phi$ is a representative of $v$. For $x\in \T\setminus T_1$ set $\varphi_x=\varphi\restriction _{\T_x}$, which is actually just $J_{x,\varphi(x)}$.

We construct the element $b$ explicitly. Below we define, for every $x\in \LL T_1$, an automorphism $b_x\in \Aut(\T_x)$, and we set $b$ to be the unique automorphism of $\T$ such that $b\restriction_{T_1}=id_{T_1}$ and $b\restriction_{\T_x}=b_x$ for $x\in \LL T_1$.
(Observe we are abusing notation here a little, since at first $b_x$ is not the restriction of an automorphism b to $\T_x$, but just an automorphism of $\T_x$; of course, once we finish the construction, it will follow that $b_x$ is the restriction of $b$ to $\T_x$.)
Obviously, this defines a unique element $b\in \Fix(T_1)$. Moreover, our construction of $b_x$ will guarantee that $b\in \Fix(T_1\cup T_2)$ and that $b=v^{-1}bav$. 

Let us discuss the latter equation in a little more detail.
Since $\bigsqcup_{x\in \LL T_1} \T_x = \T \setminus T_1$, it is enough to show that  
$b\restriction_{\T_x}=(v^{-1}bav)\restriction_{\T_x}$ is satisfied for every $x\in\LL T_1$.
However, as for every $x\in \LL T_1$ we have that $v\restriction_{\T_x}=J_{x,\varphi(x)}$ and as both $a$ and $b$ preserve the tree $\T_{\varphi(x)}$, we get $(v^{-1}bav)\restriction_{\T_x}=J_{\varphi(x),x}b_{\varphi(x)} a_{\varphi(x)} J_{x,\varphi(x)}$. That is, we need to make sure that

\begin{equation}\label{eq:I_x}
    b_x=J_{\varphi(x),x}b_{\varphi(x)} a_{\varphi(x)} J_{x,\varphi(x)}. \tag{$I_x$}
\end{equation}

is satisfied for all $x\in \LL T_1$ when we construct $b_{x}$. Now, observe that if $\varphi(x)\in\LL T_1$, then $b\restriction_{\T_{\varphi(x)}}$ is just $b_{\varphi(x)}$. If on the other hand $\varphi(x)\notin\LL T_1$, then there are two options.
Either $\varphi(x)$ has an ancestor $z$ that is a leaf of $T_1$, in which case $b\restriction_{\T_{\varphi(x)}}$ is $b_z\restriction_{\T_{\varphi(x)}}$, or $\varphi(x)$ is a root of a component $M$ of $T_1\backslash T_2$, in which case $b\restriction_{\T_{\varphi(x)}}$ is the identity on $M$ and is equal to $b_x$ on $\T_x$ for every $x\in\LL M$.
We will keep this in mind as we construct each $b_x$.

Since $[\kappa,T_1,T_2]$ is a revealing pair, every leaf of $T_1$ belongs to a maximal chain that is either an attractor, a wandering, a repeller or a periodic chain (see Remark \ref{rem:four_types_of_leafs}). Observe that if a chain $(x_{0},\dots,x_{n})$ is not periodic, then every $x_i$ except $x_n$ is a leaf of $T_1$ (and $x_n$ is never a leaf of $T_1$), whereas, if it is periodic, every $x_i$ is a leaf of $T_1$. We now explain how to construct $b_x$ for every type of leaf. 

The case of periodic chains is easy. By assumption, $v$ and $a$ act trivially there, so we can just set $b_x = id$ there for all periodic leaves $x \in \LL T_1$.

We then take care of attractor chains. Let $(s_0,\dots,s_n)$ be an attractor chain. Since $s_1,\dots,s_n\in\LL T_2$, $a$ fixes them. The vertex $s_n$ is a descendant of $s_0$, which means that $\T_{s_{n}}\subseteq\T_{s_0}$.
We wish to define $b_{s_i}\in \Aut(\T_{s_i})$ for every $i=0,\dots,n-1$, such that Eq. \ref{eq:I_x} is satisfied for $x=s_0,\dots,s_{n-1}$. That is, we need 

\begin{align}
b_{s_0}  & =J_{s_1,s_0}b_{s_1} a_{s_1} J_{s_0,s_1} \tag{$A_0$} \label{eq:A_s0}\\
b_{s_1} & =J_{s_2,s_1}b_{s_2} a_{s_2} J_{s_1,s_2} \tag{$A_1$} \label{eq:A_s1} \\
&\vdots \nonumber\\
b_{s_{n-1}}&=J_{s_n,s_{n-1}}b_{s_n} a_{s_n} J_{s_{n-1},s_n}. \tag{$A_{n-1}$}\label{eq:A_s_n-1}
\end{align}

Now, substituting $b_{s_{n-1}}$ from the last equation into the penultimate one, then substituting $b_{s_{n-1}}$ from that equation into the one before and so on, we get

\begin{equation} \label{eq:A}
    b_{s_0}=(J_{s_1,s_0}b_{s_1}J_{s_2,s_1}\cdots J_{s_n,s_{n-1}})b_{s_n}(a_{s_n}J_{s_{n-1},s_n}\cdots J_{s_1,s_2}a_{s_1}J_{s_0,s_1}). \tag{$A$}
\end{equation}

The last equation involves $b_{s_0}$ twice, because $b_{s_n}=b_{s_0}\restriction_{\T_{s_n}}$.
It is now Lemma~\ref{lem: constructing self periodic element} that will ensure us the existence of an element $b_{s_0}$ solving this equation, we explain how: Consider the right hand side of Eq. \ref{eq:A}. Since $a$ fixes $s_1,\dots,s_n$, the expression in the right parentheses is a map from $\T_{s_0}\to \T_{s_n}$, let us call it $\eta$, as in the notations of the claim; similarly, the expression in the left parentheses is a map from $\T_{s_n}\to \T_{s_0}$, let us call it $\xi$, as in the notations of the claim.
Set $b_{s_0}=\beta$ be the unique element obtained from the lemma, satisfying
\begin{align*}
    b_{s_0} \restriction _{\T_{s_0}\setminus \T_{s_n}}&= id_{\T_{s_0}\setminus \T_{s_n}} \\
    b_{s_0}\restriction _{\T_{s_n}}&= \eta b_{s_0} \xi.
\end{align*}
In particular, Eq. \ref{eq:A} is satisfied, and all leaves of the attracting component are fixed by $b_{s_0}$. 
We now set $b_x$ on the remaining leaves in the chain. Equation \ref{eq:A_s0} depends only on $b_{s_0}$ and $b_{s_1}$. Set $b_{s_1}$ such that this equation is satisfied; next, Eq. \ref{eq:A_s1} depends only on $b_{s_1}$ and $b_{s_2}$, set $b_{s_2}$ such that this equation is satisfied; and so forth, we continue until defining $b_{s_{n-1}}$ based on Eq. \ref{eq:A_s_n-1}.
Perform this process on every attractor chain.

Next we deal with leaves belonging to wandering chains. Let $(w_0,\dots,w_n)$ be a wandering chain. Again $w_0,\dots,w_{n-1}\in \LL T_1$, while $w_n$ is a leaf of a component of $T_2\setminus T_1$.
The root $s$ of this component is, as our pair is revealing, the first vertex in an attractor chain. It follows that $b_s$ was already set in the previous step, and since $\T_{w_n}\subset \T_s$, so was $b_{w_n}=b_s\restriction _{\T_{w_n}}$.
We need to define $b_{w_i}\in \Aut(\T_{w_i})$ for $i=0,1,\dots,n-1$ such that Eq. \ref{eq:I_x} holds for the leaves $x=w_0,\dots,w_{n-1}$. As above, we have to satisfy

\begin{align}
    b_{w_0}&=J_{w_1,w_0}b_{w_1} a_{w_1} J_{w_0,w_1} \tag{$W_0$}\label{eq:W_w0} \\
    b_{w_1}&=J_{w_2,w_1}b_{w_2} a_{w_2} J_{w_1,w_2} \tag{$W_1$}\label{eq:W_w1} \\
\vdots \nonumber \\
    b_{w_{n-1}}&=J_{w_n,w_{n-1}}b_{w_n} a_{w_n} J_{w_{n-1},w_n}. \tag{$W_{n-1}$}\label{eq:W_w_n-1}
\end{align}
Similar to what we had in the previous step, also here Eq. \ref{eq:W_w_n-1} depends only on $b_{w_{n-1}}$ and $b_{w_n}$. As $b_{w_n}$ is already set, we take $b_{w_{n-1}}$ to be the (unique) element satisfying this equation. We then successively set $b_{w_{n-2}},\dots,b_{w_0}$ on the same way. Perform this process on every wandering chain.

Next we come to repeller chains. Let $(r_0,\dots,r_n)$ be a repeller chain. In this case $r_0,\dots,r_{n-1}\in \LL T_1$ and $r_n$ is an ancestor of $r_0$. Moreover, $r_n$ is the root of a component $M$ of $T_1\setminus T_2$, and the leaves of this component are all vertices of wandering chains (again, since the tree pair is revealing). In particular, $b_x$ was already defined in the previous step for all $x\in \LL M\setminus \{r_0\}$.
Also here, we have to satisfy Eq. \ref{eq:I_x} for all leaves $x=r_0,\dots,r_{n-1}$. This means again
\begin{align}
    b_{r_0}&=J_{r_1,r_0}b_{r_1} a_{r_1} J_{r_0,r_1} \tag{$R_0$}\label{eq:R0} \\
    b_{r_1}&=J_{r_2,r_1}b_{r_2} a_{r_2} J_{r_1,r_2} \tag{$R_1$}\label{eq:R1} \\
&\vdots \nonumber \\
    b_{r_{n-1}}&=J_{r_n,r_{n-1}}b_{r_n} a_{r_n} J_{r_{n-1},r_n}. \tag{$R_{n-1}$}\label{eq:Rn-1}
\end{align}
Plugging in $b_{r_{n-1}}$ from the last equation to the one before, and so on, as in the attractor chain case, we get

\begin{equation}\label{eq:R}
    b_{r_0}=(J_{r_1,r_0}b_{r_1}J_{r_2,r_1}\cdots J_{r_n,r_{n-1}})b_{r_n}(a_{r_n}J_{r_{n-1},r_n}\cdots J_{r_1,r_2}a_{r_1}J_{r_0,r_1}). \tag{R}
\end{equation}

Note that $r_0=r_n\restriction_{\T_{r_0}}$. In order to define $b_{r_n}$ we use Lemma \ref{lem: constructing self periodic element} again. 
Set $\alpha := b\restriction _{\T_{r_n}\setminus \T_{r_0}}\in \Aut(\T_{r_n}\setminus \T_{r_0})$. That is, $\alpha$ fixes all leaves of the component $M$, and equals to $b_x$ for every $x\in \LL M\setminus \{r_0\}$. Set $\xi$ to be the expression that appears on the right parenthesis in Eq. \ref{eq:R}, and $\eta$ to be the expression on the left. Indeed, $\xi \colon \T_{r_0}\to \T_{r_n}$ and $\eta \colon \T_{r_n}\to \T_{r_0}$. Lastly, set $b_{r_n}=\beta$ be the unique element provided in the lemma, satisfying
\begin{align*}
    b_{r_n}\restriction_{\T_{r_n}\setminus \T_{r_0}}&=\alpha \\
    b_{r_n}\restriction_{\T_{r_0}}&= \eta b_{r_n} \xi.
\end{align*}
Indeed, such a choice satisfies Eq. \ref{eq:R}. To finish the construction, observe again that Eqs. \ref{eq:R0},$\dots$,\ref{eq:Rn-1} can be solved one by one as above, and \ref{eq:I_x} is satisfied for all $x$ in $\LL T_1$.

By construction, $b$ fixes all leaves of $T_1$. Furthermore, it fixes the attracting components. It follows that indeed $b\in \Fix(T_1\cup T_2)$.
\end{proof}

\begin{corollary}\label{cor:open conj class}
	The conjugacy class of a hyperbolic element is always open inside the class $\Hyp$ of hyperbolic almost automorphisms.
	In particular, an almost automorphism has open conjugacy class if and only if it is hyperbolic with full support.
\end{corollary}

\begin{proof}
    Let $g$ be a hyperbolic almost automorphism. Let $P = \tp{g}{T_1}{T_2}$ be a revealing tree pair associated to $g$.
    Consider the open neighborhood $U$ of $g$ consisting of all elements $f \in \AAutT$ such that $P$ is a tree pair associated to $f$.
    By Lemma \ref{lem:leaves of a revealing pair}, $U\cap \Hyp$ contains only elements $f\in U$ that are trivial on $\eell(g)$. Proposition \ref{prop:v and av are conj} implies that all such elements are conjugate to $g$.
    
    Now we prove the second part of the corollary.
    Indeed in case $g$ is hyperbolic with full support, it is obvious that all elements in $U$ are also hyperbolic with full support and, by Proposition \ref{prop:v and av are conj} they are conjugate to $g$.
    On the other hand, assume that $\eell(g)$ is non-empty.
    Recall the EH-decomposition, $g=g_eg_h$, from Definition \ref{def:EH decomposition}.
    It is, using Lemma \ref{lemma:dk-type iff OT of AAutTdk}, not difficult to find a sequence $(a_n)_n \to g_e$ of elliptic almost automorphisms such that $\supp(a_n) \subset \eell(g)$ for every $n$, but $a_n$ is not conjugate to $g_e$ for any $n$. For example, if $g_e$ has an infinite orbit on $\dT$ one can take all $a_n$ to be Higman--Thompson elements, which will force them to have only finite orbits; if $g_e$ has only finite orbits, one can take $a_n$ to have an infinite orbit.
    Clearly $a_n = (g_h a_n)_e$ and $(g_h a_n) \to g$.
    By Proposition \ref{prop:conj_can_be_checked_on_ell_and_hyp_seperately} none of $g_h a_n$ is conjugate to $g$, so the conjugacy class of $g$ is not open.
    \end{proof}

\begin{lemma}\label{lem: conj_into_V_by_V}
	Let $v \in V_{d,k}$ be induced by a tree pair $\tp{v}{T_1}{T_2}$
	and let $a,b \in \Fix(T_{1}\cap T_{2})$ be such that $b v a \in V_{d,k}$.
	Then, there exist $\widetilde{a}, \widetilde{b} \in V_{d,k} \cap \Fix(T_{1}\cap T_{2})$ such that $b v a = \widetilde{b} v \widetilde{a}$.
	Furthermore $\widetilde{a}$ and $\widetilde{b}$ can be chosen such that $\supp(\widetilde{a})\subseteq \supp(a)$ and $\supp(\widetilde{b})\subseteq \supp(b)$.
\end{lemma}

\begin{proof}
	Since $b v a \in V_{d,k}$ there exist finite complete subtrees $T_i^+$ containing $T_i$ such that the tree pair $\tp{bva}{a^{-1}(T_1^+)}{b(T_2^+)}$ induces $b v a$. Let $\widetilde{a}$ be induced by the tree pair $\tp{a}{a^{-1}(T_1^+)}{T_1^+}$ and $\widetilde{b}$ by the tree pair $\tp{b}{T_2^+}{b(T_2^+)}$.
	Note that $\tp{v}{T_1^+}{T_2^+}$ induces $v$.
	Then, clearly both $b v a$ and $\widetilde{b} v \widetilde{a}$ admit $\tp{bva}{a^{-1}(T_1^+)}{b(T_2^+)}$ as an associated tree pair, so they have to be the same element of $V_{d,k}$.
	
	Note that $\widetilde{a}$ and $\widetilde{b}$ were constructed such that $\supp(\widetilde{a})\subseteq \supp(a)$ and $\supp(\widetilde{b})\subseteq \supp(b)$.
\end{proof}

\begin{remark}
	It is in general not true that if $a=b^{-1}$ then also $\widetilde{a}$ can be chosen to equal $\widetilde{b}^{-1}$.
	That is, $V_{d,k}$ elements which are conjugate inside $\AAutT$ are not necessarily conjugate in $V_{d,k}$. 
	Otherwise Proposition \ref{prop:v and av are conj} would contradict Theorem \ref{thm:belkmatucci} by Belk and Matucci,
	as illustrated by the example in Fig. \ref{fig:conjugate in N but not in V}.
\begin{figure}%
    \centering
    \subfloat[a revealing pair for an element $v \in V$]{\includegraphics[scale=0.8]{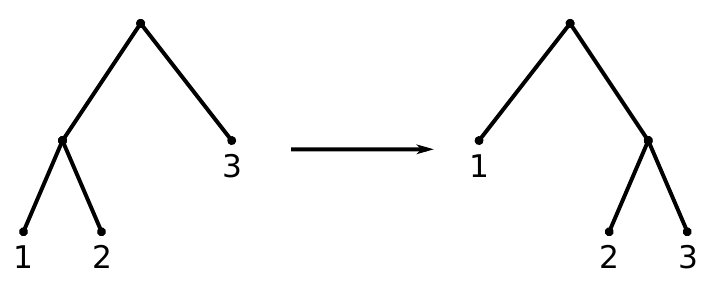}}
    \qquad
    \subfloat[its reduced BM-diagram]{\includegraphics[scale=0.8]{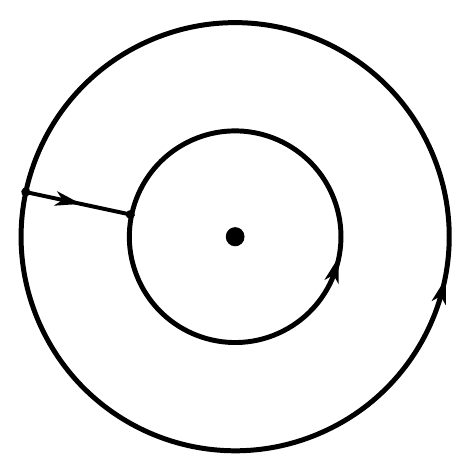}}

    \subfloat[a revealing pair for $av$]{\includegraphics[scale=0.8]{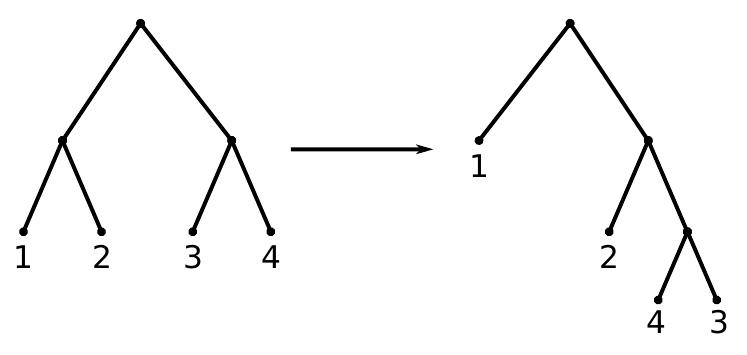}}
    \qquad
    \subfloat[its reduced BM-diagram]{\includegraphics[scale=0.8]{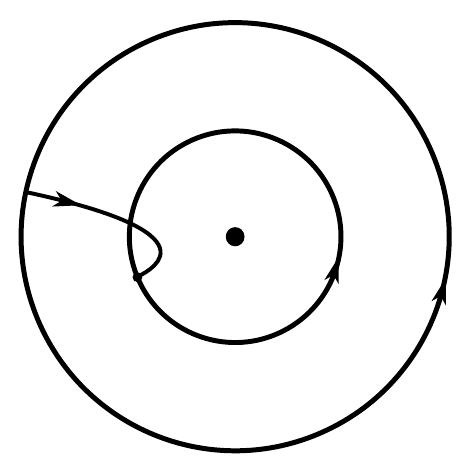}}
    \caption{The elements are not conjugate in Thompson's $V$ (Theorem \ref{thm:belkmatucci}), but they are conjugate in $\AAut(\T_{2,2})$ (Theorem \ref{thm:hyperblolic_conjugacy}).} 
\label{fig:conjugate in N but not in V}
\end{figure}
\end{remark}

\subsection{Going to diagrams and releasing rotation}

In the current subsection we complete the proof of Theorem \ref{thm:hyperblolic_conjugacy}.
For the following lemma, recall that a Higman--Thompson element $v$ is induced by a tree pair $P=[\kappa,T_1,T_2]$ if it is represented by the unique plane order preserving forest isomorphism $\foris{\phi}{T_1}{T_2}$ with $\phi|_{\LL T_1}=\kappa$. This is a stronger condition than to simply say that $P$ is a tree pair associated to $v$.

\begin{lemma}\label{lem: fixator elements change only the rotation system}
   Let $v\in V_{d,k}$ be a hyperbolic element and let $P=\tp{v}{T_1}{T_2}$ be a revealing tree pair inducing $v$. Let $a,b\in \Fix(T_1\cap T_2)$ be such that $a|_{\eell (v)}=b|_{\eell (v)}=id $, and suppose that $bva^{-1}\in V_{d,k}$. 
   Let $Q$ be a tree pair inducing $bva^{-1}$. Then, the *-reduced BM-diagrams of $P$ and $Q$ are isomorphic up to rotation.
\end{lemma}

\begin{proof}
By Lemma \ref{lem: conj_into_V_by_V} we can assume that $a,b\in V_{d,k}$. We further assume that $a=id$. The case $b=id$ works completely analogously, and clearly the lemma follows from putting together those two cases.
    Recall from Theorem \ref{thm:belkmatucci} that the reduced BM-diagrams of $P$ and $Q$ only depend on $v$ and $bva^{-1}$. Thus we can assume that the tree pair $Q=\tp{bv}{T_1^+}{T_2^+}$ satisfies $T_1 \subset T_1^+$.
    
    \emph{Step 1:} We assume first that $b$ is induced by the tree pair $\tp{b}{T_2}{b(T_2)}$.
    
    Let $D$ be the following modification of a basic BM-diagram of $P$, see Definition \ref{def:basic BM-diagram}: Instead of doing a Type II reduction on $T_1 \cap T_2$ in Step 6, we only do a Type II reduction
    on the edge connecting the root of $T_2$ to the root of $T_1$. Do the same with $P_b := \tp{bv}{T_1}{b(T_2)}$
    to obtain a BM-diagram $D_b$.
    Clearly $b$ induces an isomorphism from $D$ to $D_b$ respecting all hourglasses. So by Lemma \ref{lem: diff_in_rot_system_preserved_under_II-reduction} and Lemma \ref{lem:basic_BM_diagram_of_revealing_pair_II-reduced} the *-reduced BM-diagrams of $P$ and $P_b$ are isomorphic up to rotation.
    
    \emph{Step 2:}
    For an induction proof, set $T_1^{(0)} := T_1$ and $T_2^{(0)} := T_2$.
    For $i \geq 0$ let $S^{(i)}$ be a cancelling tree of the tree pair $P^{(i)}_b := \tp{bv}{T_1^{(i)}}{b(T_2^{(i)})}$ that is intersecting $T_1^+ \cup T_2^+$ non-trivially. 
    Let $P_b^{(i+1)}:=\tp{bv}{T_1^{(i+1)}}{b(T_2^{(i+1)})}$ be a $bv$-rolling of $P_b^{(i)}$. As usual we mean a forward rolling except in the case of a repeller chain.
    Since $T_1^+ \cup T_2^+$ is finite, there exists an $i_0 \geq 0$ such that $T_1^{(i_0)} \supset T_1^+$, and so the process of defining new tree pairs stops.
    Then $bv$ is induced by the tree pair $P_b^{(i_0)}$ and by Theorem \ref{thm:belkmatucci} the reduced BM-diagrams of $P_b^{(i_0)}$ and $Q$ are isomorphic. We will show that for each $0 \leq i < i_0$ the reduced BM-diagram of 
    $P_b^{(i)}$ is, up to rotation, isomorphic to the reduced BM-diagram of $P_b^{(i+1)}$.
    
    We define yet another tree pair. Let $v_i$ be the Higman--Thompson element induced by $P_b^{i}$ and let $P^{(i+1)}$ be the $v_i$-rolling of $P_b^{(i)}$ with $S^{(i)}$.
    By Theorem \ref{thm:belkmatucci} the reduced BM-diagram of $P^{(i+1)}$ is isomorphic to the reduced BM-diagram of $P_b^{(i)}$ because those tree pairs induce the same Higman--Thompson element $v_i$.
    Now note that setting $b_{i+1}$ to be the Higman--Thompson element induced by $\tp{b}{T_2^{(i+1)}}{b(T_2^{(i+1)})}$, we are in the situation of Step 1 with $v$ replaced by $v_i$, $P$ replaced by $P^{(i+1)}$, $b$ replaced by $b_{i+1}b_i^{-1}$ and $Q$ replaced by $P_b^{(i+1)}$. So using Step 1 we deduce that the reduced BM-diagram of $P^{(i+1)}$ is isomorphic up to rotation to the reduced BM-diagram of $P_b^{(i+1)}$. This finishes the proof.	\end{proof}

\begin{proposition}
\label{prop: diagram_differ_in_rot_sys_means_thompson_elements_differ_little}
	Let $D_1$,$D_2$ be two reduced BM-diagrams of degree $d$ that are isomorphic up to rotation, and such that for some $k\leq d-1$ they both admit a $k$-admissible cutting class.
	Then there exist revealing tree pairs $P_i=\tp{v_i}{T_1^i}{T_2^i}$ for $i=1,2$ such that $D_i$ is the basic BM-diagram of $P_i$ and such that the Higman--Thompson elements $v_1$ and $v_2$ induced by $P_1$ and $P_2$ satisfy $v_2 = av_1$ for some $a \in \Fix(T^1_1 \cap T^1_2)$. 
	
	In particular, $v_1$ and $v_2$ are conjugate.
\end{proposition}

\begin{proof}
    By Proposition \ref{prop: a reduced BM-diagram comes from a revealing pair}, there exist tree pairs $P_i=[\kappa_i,T_1^i,T_2^i]$, $i=1,2$, such that $D_i$ is the basic BM-diagram of $P_i$.
    Let $v_i$ be the Higman--Thompson element induced by $P_i$.
    For $i=1,2$ the vertices in the intersection $T_1^i\cap T_2^i$ form an hourglass in the BM-diagram generated by $P_i$ and therefore they 
    are not seen in the reduced diagram $D_i$.
    However, the isomorphism between $D_1$ and $D_2$ implies that $|\LL(T_1^1\cap T_2^1)|=|\LL(T_1^2\cap T_2^2)|$. 
    Let $P_2'$ be the tree pair obtained by replacing $T_1^2\cap T_2^2$ in $P_2$ by $T_1^1\cap T_2^1$.
    As the hourglass corresponding to $T_1^1\cap T_2^1$ is anyway subject to a Type II reduction, $P_2'$ has the same basic BM-diagram as $P_2$.
    It follows that without loss of generality we can assume that $P_2=P_2'$, namely, that $T_1^1\cap T_2^1=T_1^2\cap T_2^2$.
    
    Extend the isomorphism between $D_1$ and $D_2$ to an isomorphism between the tree pairs $P_1$ and $P_2$ that maps $T_1^1\cap T_2^1$ identically on $T_1^2\cap T_2^2$.
    Let $v_1$ and $v_2$ be the Higman--Thompson elements induced by $P_1$ and $P_2$ respectively.
    The isomorphism between $P_1$ and $P_2$ can be realized as the multiplication of $v_1$ by an element $a\in\Fix (T_1^1\cap T_2^1)$, and so the first part of the Proposition is proved.
    
    For the "in particular" part, recall that by Proposition \ref{prop:v and av are conj}, $av_1$ and $v_1$ are conjugate. 
\end{proof}

\begin{proof}[Proof of Theorem \ref{thm:hyperblolic_conjugacy}]
	By Proposition \ref{prop:v and av are conj}
	and Propsition \ref{lem: fixator elements change only the rotation system}
	we can assume that $g$ and $h$ are elements of $V_{d,k}$ without changing conjugacy classes or rotation systems of BM-diagrams.
	
	We first prove the "if"-direction. Let $D,D'$ be the reduced BM-diagrams of $g,h$ and assume that $D$ and $D'$ are isomorphic up to rotation.
	By Proposition \ref{prop: diagram_differ_in_rot_sys_means_thompson_elements_differ_little} there exist conjugate elements $v,v'$ with BM-diagrams $D,D'$. By Theorem \ref{thm:belkmatucci}, $v$ is conjugate to $g$ and $v'$ to $h$. Thus $g$ and $h$ are conjugate.
	
	Now we proof the "only if"-direction. Let $g,h$ be conjugate.
	Let $D,D'$ be their *-reduced CADSs from revealing pairs.
	By Proposition \ref{prop: diagram_differ_in_rot_sys_means_thompson_elements_differ_little} there exist elements $v,v'$ in $V_{d,k}$ that have *-reduced BM-diagrams $D$ and $D'$, and such that they differ only by an element in the fixator of the intersection of the trees of a revealing pair. By Proposition \ref{lem: fixator elements change only the rotation system} this implies that $D$ and $D'$ are isomorphic up to rotation.
\end{proof}

\begin{remark}
	This gives us the following procedure to determine the conjugacy class of a hyperbolic element.
	\begin{enumerate}
		\item Find a revealing pair representing $h$ (see Lemma \ref{lem: constructing a revealing pair}).
		\item Form the basic BM-diagram of the revealing pair (see Section \ref{from tree pairs to strand diagrams}).
		\item *-reduce the BM-diagram and forget the rotation system.
	\end{enumerate}
\end{remark}

We expect that the first step could be omitted by doing reductions similar to those considered by Aroca (see Definitions 3.9 and 3.10 in \cite{aro18}), but we decided not to pursue this idea further.

\subsection{Reading off the dymanics}\label{subsec: reading off dynamics}

In this section we explain how to read off dynamics from a *-reduced BM-diagram. This is a generalization of Theorem 5.2 and Corollary 5.3 in \cite{bema14}.
In addition we investigate when a hyperbolic element can be conjugated into $\AutT$.

\begin{theorem}\label{thm:read off dynamics}
	Let $g \in \AAutTdk$ and let $(D,r,c)$ be the *-reduced BM-diagram of a revealing pair associated to $g$.
	\begin{enumerate}
		\item Every merge loop $\mu$ with $n$ merges corresponds to an attracting point of $g$ of attracting length $n$ and period $c(\mu)$.
		\item Every split loop $\sigma$ with $n'$ splits corresponds to a repelling point of $g$ of repelling length $n'$ and period $c(\sigma)$.
		\item Every connected component of $D$ corresponds to a clopen $g$-invariant subset of $\dT$.
		\item Let $\lambda_1,\dots,\lambda_n$ be the free loops in $D$, then the number of balls $\eell(g)$ can be partitioned into is congruent to $c(\lambda_1)+\dots+c(\lambda_n) \mod{d-1}$.
	\end{enumerate}
\end{theorem}

\begin{proof}
    Clearly the theorem is true for the element $h$ constructed from $D$ in the proof of Proposition \ref{prop: a reduced BM-diagram comes from a revealing pair}. It is also not difficult to see that attracting lenghts, periods, etc. are invariant under conjugation. Since $g_h$ is conjugate to $h_h$ by Theorem \ref{thm:hyperblolic_conjugacy}, we are done.
\end{proof}

Note also that the subgraphs connecting split loops to merge loops indicate how wandering points are travelling from the repelling to the attracting points.

The following corollary rises from the classical fact that translations in a regular tree are conjugate if and only if their translation lengths agree.

\begin{corollary}\label{cor:hyperbolic conjugate to AutT}
    Let $\T=\T_{d,d+1}$ be the (non rooted) $d$-regular tree.
	A hyperbolic element in $\AAutT$ is conjugate to a translation of translation length $n$ in $\AutT$ if and only if the *-reduced BM-diagram $(D,c,r)$ of one, and hence every, revealing pair has the following form.
	The graph $D$ consists of exactly one split loop $(e_1,\dots,e_n)$ consisting of $n$ edges, exactly one merge loop $(f_1,\dots,f_n)$ consisting of $n$ edges, and for every split $o(e_i)$ in the split loop all $d-1$ outgoing edges except $e_i$ end in $o(f_i)$.
	The cohomology class $c$ is represented by $\gamma \colon \Edge(D) \to \mathbb{Z}$ with $\gamma(e_1) = \gamma(f_1) = 1$ and $\gamma(e)=0$ for all other edges.
\end{corollary}

What the BM-diagram from Corollary \ref{cor:hyperbolic conjugate to AutT} looks like is illustrated in Fig.~\ref{fig:tree_transl}.
\begin{figure}
    \centering
    \includegraphics[scale=1]{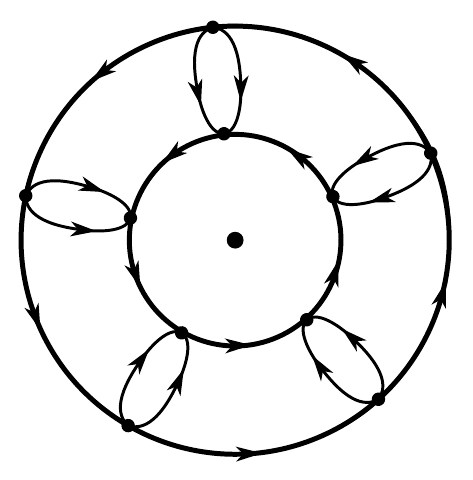}
    \caption{BM-diagram for a length $5$ translation in $\Aut(\T_{3,4})$.}
    \label{fig:tree_transl}
\end{figure}

In view of Remark \ref{rem:elliptic Aut(T) conjugates not closed}, a possible counterpart of this corollary for elliptic elements seems to be more complicated.

\bibliographystyle{alpha}
\bibliography{references}
\end{document}